\documentclass[12pt]{amsart}
\usepackage{color}
\usepackage[utf8]{inputenc}
\usepackage{mathrsfs}
\usepackage{amsfonts,amsmath}
\usepackage{amssymb,latexsym}
\usepackage{bbm}
\usepackage{mathtools}
\usepackage{etoolbox}
\patchcmd{\section}{\scshape}{\bfseries\Large}{}{}
\makeatletter
\renewcommand{\@secnumfont}{\bfseries\Large}
\def\l@subsection{\@tocline{2}{0pt}{4pc}{5pc}{}}
\makeatother

\usepackage[unicode=true,psdextra]{hyperref}
\hypersetup{
  colorlinks   = true,    
  urlcolor     = blue,    
  linkcolor    = blue,    
  citecolor    = red      
}
\usepackage[scr=boondox]{mathalpha}
\usepackage{enumitem}
\setlist[description]{leftmargin=.4cm,labelindent=0cm}
\newcommand{\curl}{\operatorname{curl}}
\newcommand{\tr}{\operatorname{tr}}
\newcommand{\dive}{\operatorname{div}}
\newcommand{\ep}{\varepsilon}

\newcommand\D{{\mathscr{D}}}
\newcommand{\uep}{u_\varepsilon}
\newcommand{\hep}{h_\varepsilon}
\newcommand{\hde}{h_\delta}

\newcommand{\jde}{J_\delta}
\newcommand{\zep}{z_\varepsilon}
\newcommand\nl[2]{\|#2\|_{L^{#1}}}
\newcommand{\om}{\omega}
\newcommand\tde{\theta_\delta}
\newcommand\omde{\omega_\delta}
\newcommand\dt{\,dt}
\newcommand\dx{\,dx}

\definecolor{bostonuniversityred}{rgb}{0.8, 0.0, 0.0}     

\newcommand{\dd}{\,{\rm d}}
\newcommand\ul{u_L}
\newcommand\hl{h_L}
\newcommand\oml{\omega_L}
\newcommand{\zl}{z_L}

\newcommand\R{{\mathbb{R}}}
\renewcommand\P{{\mathbb{P}}}
\newcommand\N{{\mathbb{N}}}

\newtheorem{theorem}{Theorem}[section]
\newtheorem{proposition}[theorem]{Proposition}
\newtheorem{lemma}[theorem]{Lemma}
\newtheorem{corollary}[theorem]{Corollary}
\theoremstyle{definition}
\newtheorem{definition}[theorem]{Definition}

\newtheorem{remark}[theorem]{Remark}
\numberwithin{equation}{section}

\usepackage{framed,color}
\definecolor{shadecolor}{rgb}{0.9,0.9,0.9}

\title[2D Micropolar Flows]{On the Role of the Viscosity Parameters in the Large Time Asymptotics of 2D Micropolar Flows}
\author{L. Brandolese, A. V. Busuioc, D. Iftimie and C. F. Perusato}

\begin{document}

\maketitle

\tableofcontents

\begin{abstract}
We investigate the role of the four viscosity parameters, in fluids where the particles possess a microstructure (micropolar flows) and are allowed to rotate in a two-dimensional setting.
We first establish the existence of global finite energy solutions, satisfying the classical energy equality, for arbitrary initial data in $L^2$, in the case of a spin viscosity $\gamma\ge0$, 
and we construct the asymptotic profiles of the solution as $t\to+\infty$. 
We deduce the remarkable fact that the large time behavior only depends on the kinematic viscosity $\mu$, and not on the other parameters $\chi$ (vortex-viscosity), $\gamma$ (spin viscosity) and $\kappa$ (gyroviscosity) of the model. 
Our primary tool is a new enstrophy-like identity of independent interest, involving the difference between the fluid vorticity and the micro-angular velocity. 
Another consequence of our analysis is the identification of scenarios where the presence of micro-rotational effects significantly enhances dissipation, thereby slowing down the fluid motion at large times. 
\end{abstract}

\section{Introduction}
Introduced by Eringen in the 1960s \cite{Eringen}, micropolar fluids constitute a notable generalization of the classical Navier-Stokes equations. They represent a significant class of non-Newtonian fluids, distinguished by particles that possess a rigid microscopic structure. This structure enables the particles to rotate about their center of mass (micro-rotation), with an angular velocity denoted by 
$W$. The conservation of linear and angular momentum, combined with certain natural assumptions regarding the stress tensor, leads to the following equations in $\R^3$:
    \begin{equation}
\label{MP}
\left\{
\begin{aligned}
&U_t+U\cdot \nabla U+\nabla P=(\mu+\chi)\Delta U+2\chi\nabla\times W, \qquad 
x\in\R^3, \;t>0\\
&W_t+(U\cdot\nabla)W=\gamma\Delta W+\kappa\nabla(\nabla\cdot W)+2\chi\nabla\times U-4\chi W,\\
&\nabla\cdot U=0,
\end{aligned}
\right.
\end{equation}
where the coefficients 
$ \mu $ (kinematic viscosity) 
and $ \chi $ (vortex viscosity)
are positive, 
and $ \kappa $ (gyroviscosity) and $ \gamma $ (spin viscosity)  are nonnega\-tive, 
all assumed to be constant (see G. Lukaszewicz \cite{Lukaszewicz}).

\par This paper will explore the two-dimensional case in which diffusion associated with the rotation of the particles may vanish, that is when $ \gamma \geq 0 $. The three-dimensional case is the subject of a forthcoming work. In the two-dimensional case, we set $U=(u_1,u_2,0)=(u,0)$ and $W=(0,0,h)$, where $u$ is a 2D vector field and $h$ a 2D scalar function. Then $\nabla\times U=(0,0,\curl u)$ where $\curl u=\partial_1u_2-\partial_2u_1$, and $\nabla\times W=-(\nabla^\perp h,0)$.  With these notations, the micropolar fluid equations in $\R^2$ read as follows:\begin{equation}\label{MP2D}
\left\{
\begin{aligned}
  \partial_t u+u\cdot\nabla u+\nabla p&=(\mu+\chi)\Delta u-2\chi\nabla^\perp h\\
\dive u&=0\\
  \partial_t h+u\cdot\nabla h&=\gamma\Delta h+2\chi\curl u-4\chi h.
\end{aligned}
\right.
\end{equation}

This system of equations models a wide range of phenomena observed in various complex fluids, such as suspensions, lubricants, and blood flow. Certain polymeric fluids and those with specific additives in thin films can be accurately described by this mathematical model (see \cite{Eringen, Lukaszewicz}).

The spin-inviscid case ($\gamma = 0$) corresponds to a regime where the diffusion of micro-angular velocity is negligible compared to other effects. Physically, this can be justified in systems where rotational inertia dominates over viscous dissipation at the microscale. In this case, the governing equation for micro-rotation becomes a hyperbolic transport equation, making the mathematical analysis more delicate, as we shall see. Moreover, a similar scenario is observed in ferrofluids. Experimental studies suggest that, in these systems, the effects associated with the diffusion of micro-angular velocity are extremely weak compared to other dissipative mechanisms \cite{Shliomis2021}, further supporting the idea that, in certain physical regimes, neglecting this diffusion term can be a reasonable approximation without significantly altering the qualitative behavior of the model. Let us also mention that the results developed in \cite{GRT} suggest that the analysis here in this paper could also be applied to other models in different contexts, such as the Ericksen–Leslie director theory for the dynamics of nematic liquid crystals developed in \cite{Ericksen1990}, for example.

\subsection*{Main Results}
	
	In this paragraph, we present our main results. In the next part, we will place our findings in perspective, comparing them with the existing literature. 

The existence and uniqueness theory of weak solutions of Leray's type
for the micropolar system~\eqref{MP2D} is close to that of the 2D Navier-Stokes equations,
 when $\mu,\chi,\gamma>0$, but  more delicate to extend in the spin-inviscid case $\gamma=0$.
In Section~\ref{sec:existence} we will establish the following:
\begin{itemize}
\item[(i)]
The existence of finite energy solutions, i.e., weak solutions belonging to Leray's space
arising from initial data $(u_0,h_0)\in L^2_\sigma(\R^2)\times L^2(\R^2)$. 
\item[(ii)]
The fact that any finite energy solution satisfies the classical energy equality, and belongs to $C^0(\R^+,L^2(\R^2))$. 
\end{itemize}

The above results are well known when $\gamma>0$ and classically proved following Leray's method. The case $\gamma=0$ is more complicated due to the absence of any smoothing effect on $h$. Indeed, even for the linear version of \eqref{MP2D} which will be studied in Section \ref{sect-linear}, the explicit formulas that we find in that section show that $h$ exhibits no smoothing effect as the time increases. Since we assume that $h_0$ is only square integrable, the solution $h$ is not better than square-integrable. This is not enough regularity to justify the multiplication of the equation of $h$ from \eqref{MP2D} by $h$. We are able to deal with this issue with an argument from the DiPerna-Lions theory.

	\medskip
In Section~\ref{sec:val-ide} we establish a couple of new remarkable enstrophy-like equalities, 
involving the interaction between the microrotation and the vorticity,
valid for any finite energy solution in the case $\gamma>0$ and for classical solutions in the case $\gamma\ge0$.
For example, for $\gamma=0$, we introduce the quantity 
\begin{equation*}
\psi(t):=\mu \|h(t)\|^2_{L^2}+\chi  \|h(t)-\textstyle\frac{\chi+\mu}{2\chi}\curl u(t)\|^2_{L^2}
\end{equation*} 
and our identity reads 
\begin{equation}
 \label{identity-intro}
\frac{\dd }{\dd t} \psi(t)
 +2\chi(\chi+\mu)\Big(\bigl\|(\curl u-2h)(t)\bigr\|^2_{L^2}
 + \bigl\|\nabla \bigl( h-\textstyle\frac{\chi+\mu}{2\chi}\curl u\bigr) (t)\bigr\|^2_{L^2}\Big)=0.
\end{equation}
In particular, we deduce the important information that, for classical solutions to~\eqref{MP2D},
one has $\frac{\dd}{\dd t}\psi(t)\le0$ in $\R^+$.
But the existence of classical solutions when $\gamma=0$, under the general assumptions $(u_0,h_0)\in L^2_\sigma(\R^2)\times L^2(\R^2)$ is out of reach, because of the lack of the regularizing effect in the equation for~$h$.
A crucial observation will be that, in the much larger class of finite energy solutions,
for which the identity~\eqref{identity-intro} is formal,
we can still prove that $\psi\in L^1(\R^+)$
is, almost everywhere in $\R^+$, equal to a monotonically decreasing function. The justification of this fact is not at all obvious. Indeed, the identity \eqref{identity-intro} involves the $H^1$ norm of $h$, so using it requires making $H^1$ estimates on $h$ while $h$ is square-integrable and not better! It is already difficult to make $L^2$ estimates on $h$, making $H^1$ estimates is pushing the difficulty one step further. Once we prove the monotonicity of the function $\psi$, we will be able to apply Zingano's  monotonicity method \cite{GuterresNichePerusatoZingano2023, ZinganoJMFM} and to deduce useful
information about the large time decay (see, e.g., Corollary~\ref{corgamma0} and also Section~\ref{sec-energy}).
Namely, in Section~\ref{sec-energy} we will deduce the following:
\begin{itemize}
\item[(iii)]
For $\gamma\ge0$,
finite energy solutions $(u,h)$ arising from 
$(u_0,h_0)\in L^2_\sigma(\R^2)\times L^2(\R^2)$   
satisfy, as $t\to+\infty$,
\[
\|u(t)\|_{L^2}^2\to0,\qquad
 t\|h(t)\|_{L^2}^2\to0,
\qquad\text{and}
\qquad
 t\|\nabla u(t)\|_{L^2}^2\to0.
\]
\end{itemize}
These limits were already known for $\gamma>0$. 
The two latter limits are especially delicate  to obtain rigorously in the case $\gamma=0$.
Even though the focus of the present paper is on the large time behavior,
we feel that the new enstrophy identities established in Section~\ref{sec:val-ide} are interesting for their own sake, 
and have potential applications to other problems in micropolar flows,
as they give a better insight on how microrotational effects affect the vorticity. 

\bigskip
In Section~\ref{sect-linear} we perform a careful analysis of the linear 2D micropolar system.
\begin{equation}\label{sysL-intro}
\left\{
\begin{aligned}
  \partial_t \ul&=(\mu+\chi)\Delta \ul-2\chi\nabla^\perp \hl \\
  \partial_t \hl &=\gamma\Delta \hl +2\chi\curl\ul-4\chi \hl 
\end{aligned}
\right.
\end{equation}
Our main result on the linear system
is the following theorem.
\begin{theorem}
\label{th:linear}
Let $\gamma\ge0$, $\mu,\chi>0$. Assume that $(u_0,h_0)\in L^2_\sigma(\R^2)\times L^2(\R^2)$.
 The solution of \eqref{sysL-intro} with initial data $(u_0,h_0)$ has the following asymptotic behavior in $L^2$:
\begin{equation}\label{ul11}
\big\|\ul-e^{\mu t\Delta}u_0+\frac12\nabla^\perp e^{\mu t\Delta}h_0\big\|_{L^2}\leq \frac Ct \|u_0\|_{L^2}+\frac C{t^{\frac32}} \|h_0\|_{L^2}
\qquad\forall t\geq1
\end{equation}
and 
\begin{equation*}
\big\|\hl -\frac12\curl e^{\mu t\Delta}u_0+\frac14\Delta e^{\mu t\Delta}h_0\big\|_{L^2}\leq\frac C{t^{\frac32}} \|u_0\|_{L^2}+\frac C{t^2} \|h_0\|_{L^2}
\qquad\forall t\geq1
\end{equation*}
for some constant $C$ depending only on the material coefficients.
\end{theorem}
The proof of this theorem, long but elementary, is based on precise $L^\infty$-frequency estimates. The result of Theorem~\ref{th:linear} is quite surprising, because it reveals that {\emph{only the kinematic viscosity parameter}}
$\mu$ (and not the viscosities $\chi$ or $\gamma$) plays a role in the large time behavior of the linear system, up to the second order.
It is also remarkable that the above asymptotic profile remains unchanged in the spin-inviscid case. It means that, as far as large time behavior is concerned, the main term in the equation of $\hl$ is not the diffusion term $\gamma\Delta\hl$ nor the damping term $-4\chi\hl$ which one would expect to lead to exponential decay for $\hl$. Surprisingly, the driving term is $2\chi\curl\ul$.

\bigskip
In Section~\ref{sec:decay} we study the algebraic decay problem of finite energy solutions $(u,h)$ of the nonlinear system~\eqref{MP2D} and discuss to which extent, for large time, these solutions behave asymptotically
as the solutions of the corresponding linear problem $(\ul,\hl)$. 

As it is well known, this question was addressed before for the Navier-Stokes equations and answered
in~\cite{Wie87} for finite energy solutions.
(See also \cite{GW05} for the case of 2D infinite energy solutions with localized vorticity).
For the solution $u_{\text{NS}}$ of 2D Navier-Stokes equation, the linear counterpart is the solution
of the heat equation $e^{t\Delta}{u_0}$ with the same initial data and Wiegner's theorem~\cite{Wie87}
asserts that, if $0\le \Gamma\le2$ and $\|e^{t\Delta}u_0\|_{L^2}^2=O(t^{-\Gamma})$ then
$\|u_{\text{NS}}(t)\|_{L^2}^2=O(t^{-\Gamma})$. Moreover, when $0<\Gamma<2$, then
$u_{\text{NS}}(t)\approx e^{t\Delta}u_0$, in the $L^2$-norm.

Theorem~\ref{th:decay} below addresses these issues for the 2D micropolar system.
Remarkably, as we shall see, our results for the micropolar system go beyond those of the Navier-Stokes equations:
indeed, for the same initial velocity $u_0$, the solutions $(u,h)$ of the 2D micropolar system
{\emph{can dissipate their energy considerably faster than the corresponding solution}} $u_{\text{NS}}$,
due to the coupling (see the conclusion section). More precisely, Theorem~\ref{th:decay} states these decay issues as follows.
\begin{theorem}
\label{th:decay}
Let $\mu,\chi>0$ and $\gamma\ge0$.
Let $(u,h)$ be a finite energy solution
to~\eqref{MP2D}, with initial data $(u_0,h_0)\in L^2_\sigma(\R^2)\times L^2(\R^2)$.
Then, the following conclusions hold:
\begin{itemize}
\item[i)] If $0 \le \Gamma\le 2$ and 
$\|e^{\mu t\Delta}(u_0-\frac12\nabla^\perp h_0)\|_{L^2}^2=O(t^{-\Gamma})$ as $t\to+\infty$,
then 
\[
\| u(t)\|^2_{L^2}=O(t^{-\Gamma})
\qquad
\text{and}
\qquad
\| h(t)\|^2_{L^2}=O(t^{-\Gamma-1})
\qquad\text{as\,\, $t\to+\infty$}.
\]
\item[ii)] 
Under the condition of the previous item (in the case $\gamma=0$ 
we also require the additional conditions $\int (1+|\xi|)|\widehat u_0(\xi)|\dd\xi <\infty $ and  
$\int (1+|\xi|)\,|\widehat h_0(\xi)|\dd \xi<\infty$), we have
	\begin{equation}
	    \label{diff-z}
	\|u(t)-\ul(t)\|^2_{L^2}+\|h(t)-\hl(t)\|^2_{L^2}  \lesssim
	\begin{cases}
		(1+t)^{-2\Gamma}& \text{if $0<\Gamma<1$}\\
	(1+t)^{-2}(\log(e+t))^2 &\text{if $\Gamma=1$}\\
	(1+t)^{-2}&\text{if $1<\Gamma\le 2$} .
	\end{cases}
	\end{equation}
Moreover, when $\gamma>0$, we actually get the following faster $L^2$-decay for $h-\hl$: 
\begin{equation}
    \label{diff-h}
\|h(t)-\hl(t)\|^2_{L^2} \lesssim
	\begin{cases}
	(1+t)^{-1-2\Gamma}& \text{if $0<\Gamma<1$}\\
	(1+t)^{-3}(\log(e+t))^2 &\text{if $\Gamma=1$}\\
	(1+t)^{-3}&\text{if $1<\Gamma\le 2$} .
	\end{cases}
	\end{equation}
\end{itemize} 
\end{theorem}
In Eq.~\eqref{diff-z}, the right-hand side is the same as in Wiegner's paper~\cite{Wie87}.
In the case $\gamma=0$, the presence of the decay is made possible by the presence of $\curl u$ and the damping term in the equation of $h$, which compensate for the absence of spin diffusion.

The Littlewood-Paley analysis (see, e.g.,~\cite[Chapter~2]{BahCD11}) allows us to reformulate our condition on the data
in terms of Besov spaces:
\[
\|e^{\mu t\Delta}(u_0-\textstyle\frac12\nabla^\perp h_0)\|_{L^2}^2\lesssim t^{-\Gamma}
\iff 
u_0-\textstyle\frac12\nabla^\perp h_0 \in \dot B^{-\Gamma}_{2,\infty}(\R^2).
\]
The most important novel feature of Theorem~\ref{th:decay}, is that the condition on the data to obtain the algebraic decay is put on 
$u_0-\textstyle\frac12\nabla^\perp h_0$, and not on both $u_0$ and $h_0$.
For example, if $u_0\in L^1\cap L^2_\sigma(\R^2)$ and $h_0\in L^2(\R^2)$, then Theorem~\ref{th:decay}
applies with $\Gamma=1$ without any additional condition being required for $h_0$.
In the conclusion section, we will illustrate further applications of Theorem~\ref{th:decay}.

	\subsection*{Context and Discussion}

	Regarding the available results in the literature concerning the existence of solutions and asymptotic behavior when $\gamma$ is positive or zero, we will only mention those most directly related to our work, as there are numerous results in various aspects concerning the micropolar system, especially in the case $\gamma > 0$. The case $\gamma=0$  is mathematically more challenging 
 and received less attention in the literature.
	
For instance, in the case of full dissipation, i.e., when $\gamma$ is positive, 
Galdi and Rionero \cite{Galdi-Rionero} showed that the existence theory for the system \eqref{MP2D} (and also \eqref{MP}) does not differ from the Navier-Stokes case. Specifically, they proved (see also Lukaszewicz \cite{Lukaszewicz}) the existence of global weak solutions in this case, following the methods developed by Temam \cite{Temam} and Ladyzhenskaya \cite{Ladyzhenskaya}. For the study of strong solutions, see \cite{Rojas-Medar} in the case $\gamma > 0$. 

In the case $\gamma=0$,
\cite{Dong-Zhang} establishes the existence and the uniqueness of global strong solutions assuming the initial data in $H^s(\R^2)$, with $s>2$.
In Theorem~\ref{th-existence} below, we obtain the global existence of finite energy solutions under more general
$L^2$-assumptions.

We now provide a brief overview of key results concerning asymptotic behavior. 
For the Navier-Stokes equations, this problem goes back to Leray, who raised the question
whether or not the weak solutions he constructed in his pioneering paper
\cite{Leray1934} go to zero in $L^2$ as $t \to \infty$. 
This problem was resolved by Kato \cite{Kato1984} and Masuda \cite{Masuda1984} fifty years later. Subsequently, Schonbek \cite{Sch85} obtained the first explicit decay rates by using her well-known {\it Fourier Splitting} technique, under an additional  $L^p$-condition, $1\le p<2$, on the initial data.
 The optimal decay rates were obtained in the already mentioned paper by Wiegner \cite{Wie87}.  

    For the micropolar problem \eqref{MP} (and \eqref{MP2D}), in the case $\gamma > 0$, the authors in \cite{GuterresNunesPerusato2019} showed that $\| u(t) \|_{L^2} = o(1)$ and $\| h(t) \|_{L^2} = o(t^{-1/2})$ as $t \to \infty$. Regarding results on the asymptotic behavior when $\gamma > 0$, we refer to \cite{BCFZ}, where the authors obtained the corresponding decay results in 3D for initial data in $L^1 \cap L^2_\sigma$ using the Fourier Splitting technique.
When $\gamma=0$, nothing has been done on the large time behavior until very recently:
we can mention only \cite {Guo-Jia-Dong} and \cite{Niu-Shang}, but the analysis therein is mainly formal, 
as the solutions considered in these articles are not known to possess the regularity needed
to justify all the computations. Moreover, Theorem~\ref{th:decay} covers
a wider range of decay rates and better decay exponents for the microrotation $h$
than in \cite{Guo-Jia-Dong,Niu-Shang}.

In the case $\gamma>0$, there are several results in the literature concerning the study of existence and time decay of {\it strong solutions}, including decay results of higher order derivatives.
These results do not extend to the case $\gamma=0$ (unless putting stringent regularity conditions on the data) and lie beyond the scope of the present work.

\subsection*{Notations}

As mentioned before, $(\ul,\hl)$ denote the solution of the linear
system~\eqref{sysL-intro}, with initial data $(u_0,h_0)$.
We will denote $\oml=\curl\ul$. For convenience, we shall also use the following notations throughout this work. The fluid vorticity will be denoted by $\omega:=\curl u$.  Furthermore, we shall denote $z=(u,h)$, $\zl=(\ul,\hl)$ and $z_0=(u_0,h_0)$.

The symbol $\mathbb{P}$ will denote the Leray projector, and $L^2_\sigma$ will represent the space 2D vector fields in $L^2(\R^2)$ that are divergence-free in the sense of distributions. 
In our estimates, we will often write 
$f(t)\lesssim g(t)$ to indicate there is a constant $C>0$, independent on time, such that $f(t)\le C g(t)$. The symbol $\mathcal{F}$ denotes the Fourier transform. We will also use the classical notation $H^s$ for $L^2$ based Sobolev spaces and $B^{s}_{p,q}$ for Besov spaces, and their homogeneous counterparts are denoted with a dot on top.

\section{Existence of global finite energy solutions}
\label{sec:existence}

This section aims to prove the global existence of finite energy solutions of \eqref{MP2D}, and the important fact that they verify the energy equality. 

Observe first that applying the Leray projector to the first equation of \eqref{MP2D} yields the following equivalent formulation:
\begin{equation}\label{sysNLproj}
\left\{
\begin{aligned}
  \partial_t u+\P(u\cdot\nabla u)&=(\mu+\chi)\Delta u-2\chi\nabla^\perp h\\
  \partial_t h+u\cdot\nabla h&=\gamma\Delta h+2\chi\curl u-4\chi h.
\end{aligned}
\right.
\end{equation}

The $L^2$ energy estimates for the above system of PDEs consist in multiplying the first equation by $u$, the last equation by $h$, and integrating in space. Doing some integrations by parts and using the classical cancellations $\int u\cdot\nabla u \cdot u=\int u\cdot\nabla h \cdot h=\int\nabla p\cdot u=0$ we formally obtain the following relation:
\begin{equation*}
\frac12\partial_t (\nl2u^2+\nl2h^2)+(\mu+\chi)\nl2{\nabla u}^2+\gamma\nl2{\nabla h}^2+4\chi\nl2h^2-4\chi \int_{\R^2}h\curl u=0
\end{equation*}
Observing by integrations by parts that
\begin{equation*}
\nl2{\nabla u}=\nl2{\curl u}
\end{equation*}
and integrating in time, we obtain the following a priori estimate: 
\begin{equation*}
\nl2{u(t)}^2+\nl2{h(t)}^2+2\mu\int_s^t\nl2{\nabla u}^2+2\gamma\int_s^t\nl2{\nabla h}^2+2\chi\int_s^t\nl2{\curl u-2h}^2=\nl2{u(s)}^2+\nl2{h(s)}^2.
\end{equation*}

This motivates the following definition of finite energy solutions of \eqref{MP2D}.
\begin{definition}[finite energy solutions]\label{fesol}
Assume that $\gamma\geq0$, $u_0\in L^2_\sigma(\R^2)$ and $h_0\in L^2(\R^2)$. A finite energy solution of \eqref{MP2D} with initial data $(u_0,h_0)$ is a couple $(u,h)$  such that:
\begin{enumerate}[label=\alph*)]
\item $u\in L^\infty(\R_+;L^2_\sigma(\R^2))\cap C^0_w([0,\infty);L^2_\sigma(\R^2))\cap L^2(\R_+;\dot H^1(\R^2))$;
\item $h\in L^\infty(\R_+;L^2(\R^2))\cap C^0_w([0,\infty);L^2(\R^2))\cap L^2(\R_+;\dot H^1(\R^2))$, if $\gamma>0$,\\ and $h\in L^\infty(\R_+;L^2(\R^2))\cap C^0_w([0,\infty);L^2(\R^2))$, if $\gamma=0$;
\item There exists some scalar pressure $p=p(x,t)$ such that \eqref{MP2D} holds true in the sense of the distributions;
\item $u\big|_{t=0}=u_0$ and $h\big|_{t=0}=h_0$.
\end{enumerate}
\end{definition}
\begin{remark}
Above, the notation $f\in  L^2(\R_+;\dot H^1(\R^2))$ must be understood in the sense that $\nabla f\in  L^2(\R_+\times\R^2)$.
\end{remark}
\begin{remark}
The condition that $u,h\in C^0_w([0,\infty);L^2(\R^2))$ is only given to make sense of $u(t)$ and $h(t)$ for all $t\geq0$, and especially for $t=0$. It always holds true if $u,h\in L^\infty(\R_+;L^2(\R^2))$ and verify  \eqref{MP2D} in the sense of the distributions. Indeed, solutions in the sense of the distributions of \eqref{MP2D} are easily seen to belong to $C^0([0,\infty);\D'(\R^2))$ and $C^0([0,\infty);\D'(\R^2))\cap L^\infty(\R_+;L^2(\R^2)) \subset C^0_w([0,\infty);L^2(\R^2))$. We will see below that $u$ and $h$ are actually strongly continuous in time with values in $L^2(\R^2)$. 
\end{remark}

Let $\jde=j_\delta\ast$ be a Friedrichs mollifier, that is $\jde$ is the convolution operator with the function
\begin{equation*}
j_\delta=\delta^{-2}j(\delta^{-1}x)
\end{equation*}
where $j\in C^\infty_0(\R^2)$ is nonnegative, even and $\int_{\R^2}j=1$.

We recall the following classical commutator lemma, in the spirit of the celebrated DiPerna-Lions theory, see \cite{diperna_ordinary_1989}.
\begin{lemma}\label{lions}
Let $I$ be an interval of $\R$, $v\in L^a(I;W^{1,p}(\R^n))$ a vector field and $g\in L^b(I;L^q(\R^n))$. Let $c,r$ be such that $\frac1r=\frac1p+\frac1q$, $\frac1c=\frac1a+\frac1b$ and assume that $c,r\in[1,\infty)$ and $a,b,p,q\in[1,\infty]$. We have that
\begin{equation*}
\jde\dive(vg)-\dive(v\jde g)\stackrel{\delta\to0}\longrightarrow0\quad\text{in }L^c(I;L^r(\R^n)).
\end{equation*}
\end{lemma}
\begin{proof}
If we omit the time dependence, this is \cite[Lemma 2.3]{lions_mathematical_1996}. The time-dependent version stated above follows from the dominated convergence theorem applied in $L^c(I)$ to the function
\begin{equation*}
F_\delta(t)=\nl{r}{\jde\dive(v(t)g(t))-\dive(v(t)\jde g(t))}.
\end{equation*}
Indeed, we have that $F_\delta(t)\to0$ a.e. when $\delta\to0$ as a consequence of \cite[Lemma 2.3]{lions_mathematical_1996} and, thanks to the same lemma, we also have the domination
\begin{equation*}
F_\delta(t)\leq C\|v(t)\|_{W^{1,p}}\nl{q}{g(t)}\in L^c(I).
\end{equation*}
\end{proof}

We prove now that all finite energy solutions verify the $L^2$ energy equality and are strongly continuous in time with values in $L^2$.
\begin{proposition}\label{prop-energ}
Assume that $\gamma\geq0$, $u_0\in L^2_\sigma(\R^2)$ and $h_0\in L^2(\R^2)$. Let $(u,h)$ be a finite energy solution of \eqref{MP2D} with initial data $(u_0,h_0)$. Then the following $L^2$ energy equality holds:
\begin{multline}\label{enereq}
\nl2{u(t)}^2+\nl2{h(t)}^2+2\mu\int_s^t\nl2{\nabla u}^2+2\gamma\int_s^t\nl2{\nabla h}^2
+2\chi\int_s^t\nl2{\curl u-2h}^2\\
=\nl2{u(s)}^2+\nl2{h(s)}^2,\qquad\forall\ 0\leq s\leq t<\infty.
\end{multline}
In addition, $u,h\in C^0\big([0,\infty);L^2(\R^2)\big)$.
\end{proposition}
\begin{proof}
We consider first the case $\gamma>0$. We observe that 
\begin{equation*}
L^\infty(\R_+;L^2(\R^2))\cap L^2_{loc}([0,\infty);H^1(\R^2))\hookrightarrow L^4_{loc}([0,\infty);L^4(\R^2)).
\end{equation*}
We know that $u\in L^\infty(\R_+;L^2(\R^2))\cap L^2_{loc}([0,\infty);H^1(\R^2))$ so $u\in L^4_{loc}([0,\infty);L^4(\R^2))$. Then $u\otimes u\in L^2_{loc}([0,\infty);L^2(\R^2))$ so $u\cdot\nabla u=\dive(u\otimes u)\in L^2_{loc}([0,\infty);H^{-1}(\R^2))$. Obviously we also have that $\Delta u\in L^2_{loc}([0,\infty);H^{-1}(\R^2))$ and $\nabla^\perp h\in L^\infty(\R_+;H^{-1}(\R^2))\hookrightarrow L^2_{loc}([0,\infty);H^{-1}(\R^2))$. We infer that $\partial_t u=-\P(u\cdot\nabla u)+(\mu+\chi)\Delta u-2\chi\nabla^\perp h\in L^2_{loc}([0,\infty);H^{-1}(\R^2))$. Since all the terms in the equality
\begin{equation*}
  \partial_t u+\P(u\cdot\nabla u)=(\mu+\chi)\Delta u-2\chi\nabla^\perp h
\end{equation*}
belong to $L^2_{loc}([0,\infty);H^{-1}(\R^2))$, the relation above can be multiplied by $u$ which belongs to the dual space $L^2_{loc}([0,\infty);H^{1}(\R^2))$ (see for instance \cite[Chapter 6]{lions_quelques_1969}). The same argument holds for the equation of $h$, showing that the equation of $h$ can be multiplied by $h$. We proved that the formal calculations given at the beginning of this section are rigorous, and  \eqref{enereq} follows.

\medskip

We assume now that $\gamma=0$. Due to the absence of diffusion in the equation of $h$, the argument above does not apply to $h$.

We fix $0\leq s\leq t<\infty$. The same argument as in the case $\gamma>0$ shows that the equation of $u$ can be multiplied by $u$ and integrated in space and time from $s$ to $t$. We get 
\begin{equation}\label{eneru}
\nl2{u(t)}^2-\nl2{u(s)}^2+2\,(\mu+\chi)\int_s^t\int_{\R^2}|\nabla u|^2-4\chi \int_s^t\int_{\R^2}h\,\curl u =0 .
\end{equation}

It remains to prove that the equation of $h$
\begin{equation}\label{eqh}
  \partial_t h+u\cdot\nabla h=2\chi\curl u-4\chi h
\end{equation}
can be multiplied by $h$ and integrated in space and time from $s$ to $t$. That is we want to show that
\begin{equation}\label{enerh}
\nl2{h(t)}^2-\nl2{h(s)}^2=2\int_s^t\int_{\R^2}(2\chi\curl u-4\chi h)h\quad\text{for all } 0\leq s\leq t<\infty.
\end{equation}
Indeed, adding \eqref{eneru} and \eqref{enerh} implies the required energy equality \eqref{enereq}.

We cannot directly multiply \eqref{eqh} by $h$ because there is not sufficient regularity on $h$ to do so. Indeed, the term $u\cdot\nabla h$ has a derivative of $h$ and we have no information on the gradient of $h$. Even if the derivative would be on $u$ instead of $h$, we would still not be able to proceed with the regularity at hand. More precisely, we could not multiply a term of the form $h\nabla u$ by $h$ and integrate it in space and time. Indeed, we would not have sufficient integrability in space since, as far as we know, $h$ and $\nabla u$ belong to $L^2$ only and we would need at least $L^3$ integrability in space. In fact, \eqref{enerh} follows from the DiPerna-Lions theory (see \cite{diperna_ordinary_1989}). We give another proof below. The same argument will be used later in Proposition \ref{prop-deriv} in a setting where the DiPerna-Lions theory no longer applies. 

We view \eqref{eqh} as a transport equation by the vector field $u\in L^2_{loc}([0,\infty);H^{1}(\R^2))$ with forcing 
$$
g:=2\chi\curl u-4\chi h\in L^2_{loc}([0,\infty);L^2(\R^2)).
$$

Let $\hde=\jde h$. Let $\sigma>0$ be fixed. We introduce the following function $f_\sigma$ which depends on some small positive constant $\sigma$:
\begin{equation}\label{fsig}
f_\sigma(s):\R\to\R,\qquad f_\sigma(s)= \frac{s}{\sqrt{1+\sigma s^2}}.
\end{equation}
We have that $f_\sigma$ is smooth, 1-Lipschitz, increasing and bounded by $\sigma^{-\frac12}$. Let also $F_\sigma$ the primitive of $f_\sigma$ given by
\begin{equation}\label{Fsig}
F_\sigma(s)=\frac{\sqrt{1+\sigma s^2}-1}\sigma
=\frac{s^2}{1+\sqrt{1+\sigma s^2}}.
\end{equation}
One can check that $|F_\sigma(s_1)-F_\sigma(s_2)|\leq |s_1^2-s_2^2|$ for all $s_1,s_2\in\R$.

We apply $\jde$ to \eqref{eqh}, we multiply it by $f_\sigma(\hde)$ and integrate on $[s,t]\times\R^2$. We get
\begin{equation}\label{eqhep}
\int_s^t\int_{\R^2}\partial_t\hde f_\sigma(\hde)
+\int_s^t\int_{\R^2}\jde(u\cdot\nabla h)f_\sigma(\hde)
=\int_s^t\int_{\R^2}\jde (g) f_\sigma(\hde) .
\end{equation}

We will pass to the limit $\delta\to0$ in each of the terms above. Since $\jde$ is a mollifier, we know that $\hde\to h$ in $L^2((s,t)\times\R^2)$, that $\hde(t)\to h(t)$ and  $\hde(s)\to h(s)$ in $L^2(\R^2)$. After extracting a subsequence, also denoted by $\hde$, we have that  $\hde\to h$ almost everywhere on $(s,t)\times\R^2$ and $\hde(t)\to h(t)$ and  $\hde(s)\to h(s)$  almost everywhere in $\R^2$. In addition, there exist two functions $H_1\in L^2(\R^2)$ and $H_2\in L^2((s,t)\times\R^2))$
such that $|\hde(t)|\leq H_1$, $|\hde(s)|\leq H_1$ almost everywhere in $\R^2$ and $|\hde|\leq H_2$ almost everywhere on $(s,t)\times\R^2$.

We deal with the first term in \eqref{eqhep}. We have that
\begin{equation*}
\int_s^t\int_{\R^2}\partial_t\hde f_\sigma(\hde)
=\int_s^t\int_{\R^2}\partial_t \big(F_\sigma(\hde)\big)
=\int_{\R^2}F_\sigma(\hde(t))-\int_{\R^2}F_\sigma(\hde(s)). 
\end{equation*}
Clearly $|F_\sigma(s)|\leq s^2$ for all $s$, so $|F_\sigma(\hde(t))|\leq H_1^2$ and  $|F_\sigma(\hde(s))|\leq H_1^2$. Then we can pass to the limit $\delta\to0$ above with the Lebesgue dominated convergence theorem and find that
\begin{equation*}
\int_s^t\int_{\R^2}\partial_t\hde f_\sigma(\hde)\stackrel{\delta\to0}\longrightarrow \int_{\R^2}F_\sigma(h(t))-\int_{\R^2}F_\sigma(h(s)). 
\end{equation*}

Next, since  $\hde\stackrel{\delta\to0}\longrightarrow h$ in $L^2((s,t)\times\R^2)$ and $f_\sigma$ is 1-Lipschitz, we infer that $f_\sigma(\hde)\stackrel{\delta\to0}\longrightarrow f_\sigma(h)$ in $ L^2((s,t)\times\R^2)$. We also know that $g\in L^2((s,t)\times\R^2)$, so $\jde(g)\stackrel{\delta\to0}\longrightarrow g$  in $ L^2((s,t)\times\R^2)$. So we have the convergence of the last term in \eqref{eqhep}:
\begin{equation*}
\int_s^t\int_{\R^2}\jde (g) f_\sigma(\hde)\stackrel{\delta\to0}\longrightarrow \int_s^t\int_{\R^2}g  f_\sigma(h).
\end{equation*}

The middle term in \eqref{eqhep} is of course the most difficult to deal with. We write it as follows:
\begin{equation}\label{difficult-term}
\int_s^t\int_{\R^2}\jde(u\cdot\nabla h)f_\sigma(\hde)  
= \int_s^t\int_{\R^2}\big[\jde(u\cdot\nabla h)-u\cdot\nabla\jde h\big]f_\sigma(\hde)
+\int_s^t\int_{\R^2}u\cdot\nabla \hde f_\sigma(\hde).
\end{equation}
Since $u\in L^2(s,t;H^1(\R^2))$ and $h\in L^\infty(s,t;L^2(\R^2))$, we can apply Lemma \ref{lions} to deduce that
\begin{equation*}
\jde(u\cdot\nabla h)-u\cdot\nabla\jde h\stackrel{\delta\to0}\to0 \quad\text{in }L^2(s,t;L^1(\R^2)).
\end{equation*}
But $f_\sigma(\hde)$ is uniformly bounded in $L^\infty\big((s,t)\times\R^2\big)$ (by $\sigma^{-\frac12}$). So the commutator term goes to 0:
\begin{equation*}
\int_s^t\int_{\R^2}\big[\jde(u\cdot\nabla h)-u\cdot\nabla\jde h\big]f_\sigma(\hde)\stackrel{\delta\to0}\longrightarrow 0.
\end{equation*}
The last term in \eqref{difficult-term} vanishes:
\begin{equation*}
\int_s^t\int_{\R^2}u\cdot\nabla \hde f_\sigma(\hde)
=\int_s^t\int_{\R^2}u\cdot\nabla (F_\sigma(\hde))
=-\int_s^t\int_{\R^2}\dive u\, F_\sigma(\hde)
=0.
\end{equation*}

Putting together the relations above and letting $\delta\to0$ in \eqref{eqhep} implies that
\begin{equation*}
\int_{\R^2}F_\sigma(h(t))-\int_{\R^2}F_\sigma(h(s))=\int_s^t\int_{\R^2}g  f_\sigma(h).
\end{equation*}
Finally, we let $\sigma\to0$. Since $|F_\sigma(s)|\leq s^2$, $|f_\sigma(s)\leq |s|$, $\lim_{\sigma\to0}F_\sigma(s)=\frac{s^2}2$ and $\lim_{\sigma\to0}f_\sigma(s)=s$, we can use again the Lebesgue dominated convergence theorem to obtain
\begin{equation}\label{energy-h}
\frac1{2}\int_{\R^2}|h(t)|^2-\frac1{2}\int_{\R^2}|h(s)|^2
=\int_s^t\int_{\R^2}g h.
\end{equation}
This is relation \eqref{enerh} so the proof of the energy equality \eqref{enereq} is completed.

Finally, relation \eqref{energy-h} shows that the application $t\mapsto\nl2{h(t)}$ is continuous. The fact that $h\in C^0_w([0,\infty);L^2(\R^2))$ then implies that $h\in C^0([0,\infty);L^2(\R^2))$.
Similarly, relation \eqref{eneru} implies that $t\mapsto\nl2{u(t)}$ is continuous so $u\in C^0([0,\infty);L^2(\R^2))$. The time continuity in the case $\gamma>0$ follows in the same way. This completes the proof.
\end{proof}

We prove now that there exist finite energy solutions in the sense of Definition \ref{fesol}.
\begin{theorem}\label{th-existence}
Assume that  $u_0\in L^2_\sigma(\R^2)$ and $h_0\in L^2(\R^2)$. There exists a finite energy solution $(u,h)$ of \eqref{MP2D} with initial data $(u_0,h_0)$. 
\end{theorem}
\begin{proof}
We consider first the case $\gamma>0$. In this case, the system \eqref{MP2D} looks like the Navier-Stokes equations, and the proof of the existence of solutions in the spaces mentioned in the Definition \ref{fesol}  follows with a straightforward adaptation of the proof for the Navier-Stokes equations. One can construct approximate solutions by mollification and pass to the limit by compactness arguments, see for example \cite[Chapter 17]{taylor_partial_1997}.

We consider now the case $\gamma=0$ which is more complicated due to the absence of diffusion in the equation of $h$. Let $\ep>0$. We solve the following system of PDEs:
\begin{equation}\label{approx}
\left\{
\begin{aligned}
  \partial_t \uep+\uep\cdot\nabla \uep+\nabla p_\ep&=(\mu+\chi)\Delta \uep-2\chi\nabla^\perp \hep\\
\dive \uep&=0\\
  \partial_t \hep+\uep\cdot\nabla \hep&=\ep\Delta \hep+2\chi\curl \uep-4\chi \hep
\end{aligned}
\right.
\end{equation}
with initial data $(\uep,\hep)\big|_{t=0}=(u_0,h_0)$. The solution $\zep:=(\uep,\hep)$ exists due to the case $\gamma>0$ and, thanks to Proposition \ref{prop-energ}, it satisfies the following energy equality:
\begin{multline}\label{enep}
\nl2{\uep(t)}^2+\nl2{\hep(t)}^2+2\mu\int_0^t\nl2{\nabla \uep}^2+2\ep\int_0^t\nl2{\nabla \hep}^2
+2\chi\int_0^t\nl2{\curl \uep-2\hep}^2\\
=\nl2{ u_0}^2+\nl2{ h_0}^2
\end{multline}
for all $t\geq0$. In particular, we have that $\nl2{\zep(t)}$ is bounded   independently of $t$. Since $\zep$ is bounded in $L^\infty(\R_+;L^2(\R^2))$, there exists a subsequence which converges weak$\ast$ in $L^\infty(\R_+;L^2(\R^2))$. In the sequel, we will extract several subsequences which, by abusing notation, will all be denoted by $\zep$, $\uep$ or $\hep$. Let $z=(u,h)$ be the limit of $\zep$:
\begin{equation}\label{weak-cv}
\zep\rightharpoonup z\quad\text{weak$\ast$ in }L^\infty(\R_+;L^2(\R^2)).
\end{equation}
Relation \eqref{enep} also implies that $\uep$ is bounded in $L^2_{loc}([0,\infty);H^1(\R^2))$ so we can extract a subsequence such that
\begin{equation*}
\uep\rightharpoonup u\quad\text{weakly in }L^2_{loc}([0,\infty);H^1(\R^2)).
\end{equation*}
In particular, $u\in L^2_{loc}([0,\infty);H^1(\R^2))$ and is divergence free. 

Next,  $\uep\otimes \uep$ is bounded in $L^\infty([0,\infty);L^1(\R^2))\hookrightarrow L^\infty([0,\infty);H^{-2}(\R^2))$ so $\uep\cdot\nabla \uep=\dive(\uep\otimes \uep)$  is bounded in $L^\infty([0,\infty);H^{-3}(\R^2))$. 
Then $\P(\uep\cdot\nabla \uep)$ is also bounded in $L^\infty([0,\infty);H^{-3}(\R^2))$. Similarly, the terms $\Delta \uep$ and  $\nabla^\perp \hep$ are bounded in $L^\infty([0,\infty);H^{-3}(\R^2))$. Applying the Leray projector $\P$ to the first equation in \eqref{approx} implies that 
\begin{equation*}
\partial_t \uep=-\P(\uep\cdot\nabla \uep)+(\mu+\chi)\Delta \uep-2\chi\nabla^\perp \hep \quad\text{is bounded in }L^\infty([0,\infty);H^{-3}(\R^2)).
\end{equation*}
Similarly,
\begin{equation*}
\partial_t \hep=-\uep\cdot\nabla \hep-\ep\Delta \hep+2\chi\curl \uep-4\chi \hep\quad\text{is bounded in }L^\infty([0,\infty);H^{-3}(\R^2)).
\end{equation*}
We infer that the $\zep=(\uep,\hep)$ are bounded and equicontinuous in $C^0([0,\infty);H^{-3}(\R^2))$. From the compact embedding $H^{-3}(\R^2)\hookrightarrow H^{-4}_{loc}(\R^2)$ and the Ascoli theorem we get that
\begin{equation*}
\zep\to z \quad\text{strongly in } C^0([0,\infty);H^{-4}_{loc}(\R^2)).
\end{equation*}
In particular $\zep\big|_{t=0}\to z\big|_{t=0}$ in $H^{-4}_{loc}(\R^2)$, so $u\big|_{t=0}=u_0$ and $h\big|_{t=0}=h_0$. Moreover, recalling that $\uep$ is bounded in $L^2_{loc}(\R_+;H^1(\R^2))$ we get by interpolation that 
\begin{equation}\label{strong-cv}
\uep\to u \quad\text{strongly in } L^2_{loc}(\R_+\times\R^2).
\end{equation}

We now pass to the limit in the first equation of \eqref{approx}. Due to \eqref{strong-cv}, we have that  $\uep\cdot\nabla \uep=\dive(\uep\otimes \uep)\to \dive(u\otimes u)=u\cdot\nabla u$ in the sense of the distributions. Clearly we also have that $\partial_t \uep\to 
\partial_t u$, $\Delta \uep\to\Delta u$ and $\nabla^\perp \hep\to \nabla^\perp h$ in the sense of the distributions. Recalling that the limit of a gradient in the sense of the distributions is another gradient, passing to the limit $\ep\to0$ in the first equation of \eqref{approx} yields the first equation of \eqref{MP2D}.

Now, \eqref{strong-cv} together with \eqref{weak-cv} allows us to pass to the limit in the term $\uep\hep$: we have that $\uep\hep\to uh$ in the sense of the distributions. We infer that $(\uep\cdot\nabla \hep)=\dive(\uep\hep)\to\dive(uh)=u\cdot\nabla h$ in the sense of the distributions. We also clearly have that $  \partial_t \hep\to  \partial_t h$, $\curl \uep\to \curl u$ and $\hep\to h$ in the sense of the distributions. Moreover, $\Delta\hep\to\Delta h$ in the sense of the distributions, so $\ep\Delta\hep\to0$ in the sense of the distributions. We conclude that passing to the limit in the last equation of \eqref{approx} implies the last equation of \eqref{MP2D}. This completes the proof that \eqref{MP2D} (with $\gamma=0$) is verified in the sense of the distributions. 

\medskip

We constructed a solution of \eqref{MP2D} such that
\begin{equation*}
u\in L^\infty(\R_+;L^2_\sigma(\R^2))\cap L^2_{loc}([0,\infty);H^1(\R^2))\cap C^0([0,\infty);H^{-4}_{loc}(\R^2))
\end{equation*}
and
\begin{equation*}
h\in L^\infty(\R_+;L^2(\R^2))\cap C^0([0,\infty);H^{-4}_{loc}(\R^2))
\end{equation*}
By density of compactly supported smooth functions in $L^2(\R^2)$, we have that
\begin{equation*}
L^\infty(\R_+;L^2(\R^2))\cap C^0([0,\infty);H^{-4}_{loc}(\R^2))\subset C^0_w([0,\infty);L^2(\R^2))
\end{equation*}
so
\begin{equation*}
u,h\in C^0_w([0,\infty);L^2(\R^2)).
\end{equation*}

We conclude that the solution $(u,h)$ has all the regularity required in Definition \ref{fesol}, so it is a finite energy solution. This completes the proof. 
\end{proof}

\section{A new enstrophy-like identity}
\label{sec:val-ide}

This section is devoted to the proof of a new enstrophy-like identity. The consequences of this new identity are quite important, for example, it will allow us to prove the decay of the $L^2$ norm of the solution in the case $\gamma=0$ without any additional condition on the material coefficients or on the initial data. We define the following quantity:
\begin{equation}\label{deftheta}
\theta(x,t)=h(x,t)-\frac{\chi+\mu}{2\chi}\om(x,t)
\end{equation}
where we recall that $\om=\curl u$.
We first state our new enstrophy identity under the form of a priori estimates in the case $\gamma=0$.
\begin{lemma}[new enstrophy identity]\label{lidentity}
Let $\gamma=0$ and $(u,h)$ be a sufficiently smooth solution of \eqref{MP2D}. We have the following identity:
\begin{equation}
 \label{identity}
\frac{\dd }{\dd t}\Bigl(\mu \|h(t)\|^2_{L^2}+\chi  \|\theta(t)\|^2_{L^2}\Bigr)
 +2\chi(\chi+\mu)\Big(\|(\om-2h)(t)\|^2_{L^2}
 + \|\nabla \theta(t)\|^2_{L^2}\Big)=0.
\end{equation}
\end{lemma}
\begin{proof}
Taking the $\curl$ of the first equation in \eqref{MP2D} implies the following equation for $\omega$:
\begin{equation*}
\partial_t\omega + u\cdot\nabla\omega=(\mu+\chi)\Delta\omega-2\chi\Delta h=-2\chi\Delta \theta.
\end{equation*}
Multiply the above relation by $\frac{\chi+\mu}{2\chi}$ and subtract it from the equation for $h$ given in \eqref{MP2D}. We obtain the following PDE for the quantity $\theta$:
\begin{equation}\label{eq-theta}
\partial_t \theta + u\cdot\nabla \theta=2\chi\om-4\chi h+(\chi+\mu)\Delta \theta.
\end{equation}
Let us multiply the above relation by $\theta$ and integrate in space. We get, after some integrations by parts,
\begin{align*}
\frac{1}{2}\frac{\dd }{\dd t}\|\theta\|_{L^2}^2
=\int \big(2\chi\om-4\chi h+(\chi+\mu)\Delta \theta\big)\theta
=2\chi\int\om \theta-4\chi\int h\theta-(\chi+\mu)\int|\nabla \theta|^2.
\end{align*}
Multiplying the equation for $h$ given in \eqref{MP2D} by $h$ and integrating in space yields
\begin{equation*}
\frac{1}{2}\frac{\dd }{\dd t}\|h\|_{L^2}^2
=2\chi\int\om h-4\chi\int|h|^2
\end{equation*}
Now we multiply the above relation by $\mu$, the penultimate relation by $\chi$ and add. We get
\begin{equation*}
\frac{1}{2}\frac{\dd }{\dd t}\big( \mu\|h\|_{L^2}^2+\chi \|\theta\|_{L^2}^2\big)
=2\mu\chi\int\om h-4\mu\chi\int|h|^2
+2\chi^2\int\om \theta-4\chi^2\int h\theta-\chi(\chi+\mu)\int|\nabla \theta|^2.
\end{equation*}
Given the definition of $\theta$ given in \eqref{deftheta}, one can check that
\begin{equation*}
2\mu\chi\int\om h-4\mu\chi\int|h|^2+2\chi^2\int\om \theta-4\chi^2\int h\theta=-\chi(\chi+\mu)\int |\om-2h|^2.
\end{equation*}
The conclusion follows. 
\end{proof}

Due to the lack of regularity of finite energy solutions, we cannot prove rigorously this new identity in the case $\gamma=0$ if we assume $u_0$ and $h_0$ to be square-integrable only. Indeed, if $\gamma=0$ there is no smoothing effect in the equation of $h$, even for the linear version of \eqref{MP2D}, as can be observed from the explicit expressions that we will find in Section \ref{sect-linear}. So $h$ is not better than square-integrable in space. But \eqref{identity} involves the $H^1$ norm of $\theta$ which in turn involves the $H^1$ norm of $h$. Nevertheless, as will be clear from what follows, we are merely interested in showing the monotonicity of the following function
\begin{equation*}
\psi(t):=\mu \|h(t)\|^2_{L^2}+\chi  \|\theta(t)\|^2_{L^2}.
\end{equation*}
This function is defined almost everywhere because $\omega\in L^2(\R_+\times\R^2)$. More precisely, we have the following lemma.
\begin{lemma}\label{psil1}
Assume that $(u,h)$ is a finite energy solution of \eqref{MP2D} in the sense of Definition \ref{fesol}. 
We have that $\psi\in L^1(\R_+)$.
\end{lemma}
\begin{proof}
Since $u\in L^2(\R_+;\dot H^1(\R^2))$ we have that $\omega\in L^2(\R_+\times\R^2)$. It remains to prove that $h\in L^2(\R_+\times\R^2)$. We observed during the proof of Proposition \ref{prop-energ} that $h$ satisfies the energy equality stated in relation \eqref{enerh}. Setting $s=0$ in  \eqref{enerh} yields, for all $t\geq0$,
\begin{align*}
\nl2{h(t)}^2+8\chi\int_0^t \nl2{h}^2
&=\nl2{h_0}^2+4\chi\int_0^t\int_{\R^2}\om h\\
&\leq\nl2{h_0}^2+4\chi\int_0^t\int_{\R^2}\nl2\om\nl2h\\
&\leq\nl2{h_0}^2+4\chi\int_0^t\nl2h^2+\chi\int_0^t\nl2\om^2.
\end{align*}
So
\begin{equation*}
4\chi\int_0^t \nl2{h}^2\leq\nl2{h_0}^2+\chi\int_0^t\nl2\om^2
\end{equation*}
which implies that $h\in L^2(\R_+\times\R^2)$ using once more that $\omega\in L^2(\R_+\times\R^2)$. This completes the proof.
\end{proof}

We can now show the following consequence on $\psi$ of the new enstrophy identity.
\begin{proposition}\label{prop-deriv}
Let $\gamma=0$ and $(u_0,h_0)\in L^2_\sigma(\R^2)\times L^2(\R^2)$. Let $(u,h)$ be a finite energy solution of \eqref{MP2D} (in the sense of Definition \ref{fesol}) with initial data $(u_0,h_0)$. Then $\psi'\leq0$ in $\D'\big((0,\infty)\big)$.
\end{proposition}
\begin{proof}
We proceed as in the proof of Proposition \ref{prop-energ}. Let $\varphi\in C^\infty_0\big((0,\infty);\R_+\big)$ be a nonnegative compactly supported smooth function. Let $\jde$ be a Friedrichs mollifier,  $\hde=\jde h$, $\omde=\jde\omega$ and $\tde=\jde\theta$. We use the functions $f_\sigma$ and $F_\sigma$ introduced in \eqref{fsig} and \eqref{Fsig}.	

Since $\om,\theta,h\in L^2(\R_+\times\R^2)$, we know from the properties of the mollifications that
\begin{equation}\label{othconv}
\omde\to\omega,\quad \tde\to\theta,\quad \hde\to h\quad \text{strongly in } L^2(\R_+\times\R^2).
\end{equation}

We apply $\jde$ to the equation of $\theta$ given in \eqref{eq-theta}, multiply by $f_\sigma(\tde)\varphi(t)$ and integrate in space and time. We get
\begin{align}
\int_0^\infty\int_{\R^2}\partial_t \tde f_\sigma(\tde)&\varphi(t)\dt\dx
+\int_0^\infty\int_{\R^2}\jde(u\cdot\nabla \theta)f_\sigma(\tde)\varphi(t)\dt\dx\notag\\
&=2\chi \int_0^\infty\int_{\R^2}\omde f_\sigma(\tde)\varphi(t)\dt\dx
-4\chi \int_0^\infty\int_{\R^2}\hde f_\sigma(\tde)\varphi(t)\dt\dx\label{thetadelta}\\
&\hskip 4cm +(\chi+\mu)\int_0^\infty\int_{\R^2}\Delta\tde f_\sigma(\tde)\varphi(t)\dt\dx.\notag
\end{align}

We wish to pass to the limit $\delta\to0$ above. We will deal with all the terms, one by one.

First, we write 
\begin{equation*}
\int_0^\infty\int_{\R^2}\partial_t \tde f_\sigma(\tde)\varphi(t)\dt\dx
=\int_0^\infty\int_{\R^2}\partial_t\big(F_\sigma(\tde)\big)\varphi(t)\dt\dx
=-\int_0^\infty\varphi'(t)\int_{\R^2}F_\sigma(\tde)\dx\dt
\end{equation*}
Recall that $\tde\to\theta$ strongly in $L^2(\R_+\times\R^2)$ and that  $|F_\sigma(s_1)-F_\sigma(s_2)|\leq |s_1^2-s_2^2|$. Then
\begin{multline*}
\Big|\int_0^\infty\varphi'(t)\int_{\R^2}F_\sigma(\tde)\dx\dt-\int_0^\infty\varphi'(t)\int_{\R^2}F_\sigma(\theta)\dx\dt\Big|\leq \int_0^\infty|\varphi'(t)|\int_{\R^2}|F_\sigma(\tde)-F_\sigma(\theta)|\dx\dt\\
\leq \int_0^\infty \int_{\R^2} |\varphi'(t)|\,|\tde^2-\theta^2|\dt\dx
\leq  \int_0^\infty |\varphi'(t)|\,\nl2{\tde-\theta}\nl2{\tde+\theta}\dt
\stackrel{\delta\to0}\longrightarrow0
\end{multline*}
so
\begin{equation*}
\int_0^\infty\int_{\R^2}\partial_t \tde f_\sigma(\tde)\varphi(t)\dt\dx\stackrel{\delta\to0}\longrightarrow
-\int_0^\infty\varphi'(t)\int_{\R^2}F_\sigma(\theta)\dx\dt
\end{equation*}

Next, we write
\begin{align*}
\int_0^\infty\int_{\R^2}\jde(u\cdot\nabla \theta)&f_\sigma(\tde)\varphi(t)\dt\dx\\
&=\int_0^\infty\int_{\R^2}u\cdot\nabla \tde f_\sigma(\tde)\varphi(t)\dt\dx
+\int_0^\infty\int_{\R^2}[\jde,u\cdot\nabla]\theta\,f_\sigma(\tde)\varphi(t)\dt\dx\\
&=\int_0^\infty\int_{\R^2}[\jde,u\cdot\nabla]\theta\,f_\sigma(\tde)\varphi(t)\dt\dx.
\end{align*}

Recalling that $u\in L^2_{loc}\big([0,\infty); H^1(\R^2)\big)$ and $\theta\in L^2\big(\R_+\times\R^2\big)$, we know from Lemma \ref{lions} that
\begin{equation*}
[\jde,u\cdot\nabla]\theta\stackrel{\delta\to0}\longrightarrow0\quad\text{strongly in }L^{1}_{loc}\big([0,\infty);L^1(\R^2)\big).
\end{equation*}
Since $f_\sigma(\tde)$ is bounded uniformly with respect to $\delta$, $|f_\sigma(\tde)|\leq \sigma^{-\frac12}$, we infer that
\begin{equation*}
\int_0^\infty\int_{\R^2}\jde(u\cdot\nabla \theta)f_\sigma(\tde)\varphi(t)\dt\dx
=\int_0^\infty\int_{\R^2}[\jde,u\cdot\nabla]\theta\,f_\sigma(\tde)\varphi(t)\dt\dx
\stackrel{\delta\to0}\longrightarrow0.
\end{equation*}

Since $\tde\to\theta$ strongly in $L^2(\R_+\times\R^2)$ and the function $f_\sigma$ is 1-Lipschitz, we have that $f_\sigma(\tde)\to f_\sigma(\theta)$ strongly in $L^2(\R_+\times\R^2)$. Recalling \eqref{othconv}, we immediately see that
\begin{equation*}
\int_0^\infty\int_{\R^2}\omde f_\sigma(\tde)\varphi(t)\dt\dx
\stackrel{\delta\to0}\longrightarrow
\int_0^\infty\int_{\R^2}\omega f_\sigma(\theta)\varphi(t)\dt\dx
\end{equation*}
and
\begin{equation*}
\int_0^\infty\int_{\R^2}\hde f_\sigma(\tde)\varphi(t)\dt\dx
\stackrel{\delta\to0}\longrightarrow
\int_0^\infty\int_{\R^2}h f_\sigma(\theta)\varphi(t)\dt\dx.
\end{equation*}

We are not able to pass to the limit $\delta\to0$ in the last term in \eqref{thetadelta}, but we can show that it has a good sign:
\begin{multline}\label{example}
\int_0^\infty\int_{\R^2}\Delta\tde f_\sigma(\tde)\varphi(t)\dt\dx
=-\int_0^\infty\int_{\R^2}\nabla\tde \nabla\big(f_\sigma(\tde)\big)\varphi(t)\dt\dx\\
=-\int_0^\infty\int_{\R^2}|\nabla\tde|^2f'_\sigma(\tde)\varphi(t)\dt\dx\leq0.
\end{multline}

We can now go back to \eqref{thetadelta} and take the $\limsup\limits_{\delta\to0}$. We get
\begin{equation*}
-\int_0^\infty\varphi'(t)\int_{\R^2}F_\sigma(\theta)\dx\dt
\leq 2\chi \int_0^\infty\int_{\R^2}\omega f_\sigma(\theta)\varphi(t)\dt\dx
-4\chi \int_0^\infty\int_{\R^2}h f_\sigma(\theta)\varphi(t)\dt\dx.
\end{equation*}

Now, we let $\sigma\to0$. Since $|F_\sigma(s)|\leq s^2$, $|f_\sigma(s)\leq |s|$, $\lim_{\sigma\to0}F_\sigma(s)=\frac{s^2}2$ and $\lim_{\sigma\to0}f_\sigma(s)=s$, we can pass to the limit above with the dominated convergence theorem. We obtain that
\begin{equation*}
-\frac12\int_0^\infty\varphi'(t)\int_{\R^2}|\theta|^2\dx\dt
\leq 2\chi \int_0^\infty\int_{\R^2}\omega \theta\varphi(t)\dt\dx
-4\chi \int_0^\infty\int_{\R^2}h \theta\varphi(t)\dt\dx.
\end{equation*}

One can justify similarly that we can multiply the equation of $h$ by $h\varphi(t)$ and integrate in space and time. We actually already did that in the proof of Proposition \ref{prop-energ}. We get
\begin{equation*}
-\int_0^\infty\varphi'(t)\int_{\R^2}|h|^2\dx\dt
=4\chi \int_0^\infty\int_{\R^2}\omega h\varphi(t)\dt\dx
-8\chi \int_0^\infty\int_{\R^2}|h|^2\varphi(t)\dt\dx.
\end{equation*}
Multiplying the above relation by $\mu$, the penultimate relation by $2\chi$ and adding implies, after some calculations using \eqref{deftheta}, that
\begin{equation*}
-\int_0^\infty\varphi'(t)\psi(t)\dt\leq -2\chi(\chi+\mu)\int_0^\infty\int_{\R^2}|\om-2h|^2\varphi(t)\dt\dx\leq0.
\end{equation*}
This means that $\psi'\leq0$ in $\D'\big((0,\infty)\big)$ and completes the proof.
\end{proof}

\begin{corollary}\label{corgamma0}
Let $\gamma=0$ and $(u_0,h_0)\in L^2_\sigma(\R^2)\times L^2(\R^2)$. Let $(u,h)$ be a finite energy solution of \eqref{MP2D} (in the sense of Definition \ref{fesol}) with initial data $(u_0,h_0)$. Then
\begin{equation*}
\lim_{t\to\infty}\sqrt t\nl2{h(t)}=0
\end{equation*}
and there exists a negligible set $A\subset\R_+$ such that
\begin{equation*}
\lim_{\substack{t\to\infty\\ t\not\in A}}\sqrt t\nl2{\nabla u(t)}=0.
\end{equation*}
\end{corollary}
\begin{proof}
Proposition \ref{prop-deriv} states that $\psi'\leq0$ in $\D'\big((0,\infty)\big)$. Then there exists a monotonically decreasing function $G$ defined on $\R_+$ such that $\psi=G$ almost everywhere. Let $A$ be negligible such that $\psi=G$ on $\R_+\setminus A$. From Lemma \ref{psil1} we infer that $G\in L^1(\R_+)$. We bound
\begin{equation*}
    t\,G(t) = 2\int_{t/2}^t G(t) \dd s \leq 2\int_{t/2}^t G(s) \dd s
    \to 0 \quad {\text as\,\,} t \to \infty.
\end{equation*}
We infer that
\begin{equation*}
\lim_{\substack{t\to\infty\\ t\not\in A}}t\psi(t)=0.
\end{equation*}
So
\begin{equation*}
\lim_{\substack{t\to\infty\\ t\not\in A}}\sqrt t\nl2{h(t)}=0
\quad\text{and}\quad
\lim_{\substack{t\to\infty\\ t\not\in A}}\sqrt t\nl2{\theta(t)}=0.
\end{equation*}
The function $t\mapsto \nl2{h(t)}$ being continuous, we can drop $t\not\in A$ in the first limit above which proves the first part of the corollary. Recalling that
\begin{equation*}
\nl2{\nabla u}=\nl2\omega=\nl2{\frac{2\chi}{\chi+\mu}(\theta-h)}
\leq \frac{2\chi}{\chi+\mu}(\nl2\theta+\nl2h)
\end{equation*}
completes the proof of the corollary.
\end{proof}

We consider now the case $\gamma>0$ which is much simpler and the enstrophy identity can be proved rigorously. Let us introduce the following quantity (see \cite{Cesar-RG-Zingano-CP2023} for the physical meaning of this quantity) 
\begin{equation*}
\epsilon :=h-\frac12\omega.
\end{equation*}

We have the following result.
\begin{proposition}\label{pident}
Let $\gamma>0$ and $(u_0,h_0)\in L^2_\sigma(\R^2)\times L^2(\R^2)$. Then $u\in C^0((0,\infty);H^1(\R^2))\cap L^2_{loc}((0,\infty);H^2(\R^2))$ and the following enstrophy identity holds true:
\begin{equation}
 \label{enEO22}
\frac12\frac{\dd }{\dd t}\Bigl(\|\epsilon\|_{L^2}^2+a\|\omega\|_{L^2}^2\Bigr)
 +4\chi\|\epsilon\|_{L^2}^2
 +(\gamma+\chi)\big\|\nabla\epsilon-\frac{\mu}{2\chi}\nabla\omega\big\|_{L^2}^2
 +\frac{\mu\gamma(\mu+\chi)}{4\chi^2}\|\nabla\omega\|_{L^2}^2=0
\end{equation}
where 
\begin{equation}
\label{choice-a}
 a:=\frac{\gamma\chi+\mu\chi+2\mu\gamma}{4\chi^2}.
\end{equation}
In addition,
\begin{equation}\label{smallgammapositive}
\lim_{t\to\infty}\sqrt t\nl2{h(t)}=0
\quad\text{and}\quad
\lim_{t\to\infty}\sqrt t\nl2{\nabla u(t)}=0.
\end{equation}
\end{proposition}
\begin{proof}
We know that $\nabla h\in L^2(\R_+\times\R^2)$, see relation \eqref{enereq}. The velocity field $u$ verifies the 2D Navier-Stokes equations with forcing $-2\chi\nabla^\perp h\in L^2(\R_+\times\R^2)$. The classical regularity results for the 2D Navier-Stokes equations imply that $u\in C^0((0,\infty);H^1(\R^2))\cap L^2_{loc}((0,\infty);H^2(\R^2))$.

We prove now the enstrophy identity \eqref{enEO22}. Taking the $\curl$ of the first equation in \eqref{MP2D} implies the following equation for $\omega$:
\begin{equation}\label{omega}
\partial_t\omega + u\cdot\nabla\omega=(\mu+\chi)\Delta\omega-2\chi\Delta h=\mu\Delta\omega-2\chi\Delta\epsilon.
\end{equation}
Subtracting one half of the equation above from the last equation in \eqref{MP2D} gives the following PDE verified by $\epsilon$:
\begin{equation}\label{eps}
\partial_t\epsilon + u\cdot\nabla\epsilon=2\chi\omega-4\chi h-\frac{1}{2}(\mu+\chi)\Delta\omega+(\chi+\gamma)\Delta h
=(\gamma+\chi)\Delta\epsilon+\frac{\gamma-\mu}{2}\Delta\omega-4\chi\epsilon.
\end{equation}

One can easily check that
\begin{equation*}
u\cdot\nabla\epsilon,\ u\cdot\nabla\omega,\ \epsilon,\ \Delta\omega, \ \Delta \epsilon\in L^2_{loc}((0,\infty);H^{-1}(\R^2))
\end{equation*}
so all the terms in \eqref{omega} and \eqref{eps} belong to the space $ L^2_{loc}((0,\infty);H^{-1}(\R^2))$. We infer that we can multiply these two relations by quantities that are in the space $L^2_{loc}((0,\infty);H^{1}(\R^2))$. Clearly $\omega,\epsilon\in L^2_{loc}((0,\infty);H^{1}(\R^2))$, so we can  multiply \eqref{omega} by $\omega$, \eqref{eps} by $\epsilon$ and integrate in space. We get  after a few integrations by parts:
\begin{gather*}
\frac12\frac{\dd }{\dd t}\|\epsilon\|^2+(\gamma+\chi)\int|\nabla \epsilon|^2
+\frac{\gamma-\mu}{2} \int\nabla\omega\cdot\nabla\epsilon+4\chi\int|\epsilon|^2=0\\
\intertext{and}
\frac12\frac{\dd }{\dd t}\|\omega\|^2+\mu\int|\nabla \omega|^2 -2\chi\int\nabla\omega\cdot\nabla\epsilon=0.
\end{gather*}
Multiplying by $a$ the second relation above and adding it to the first one we obtain
\begin{multline*}
\frac12\frac{\dd }{\dd t}\Bigl(\|\epsilon\|^2+a\|\omega\|^2\Bigr)
 +4\chi\int|\epsilon|^2
 + (\gamma+\chi)\int|\nabla\epsilon|^2
 +\Bigl(\frac{\gamma-\mu}{2}-2\chi a\Bigr)\int\nabla\epsilon\cdot\nabla\omega\\
 +a\mu\int|\nabla\omega|^2=0.
\end{multline*}
The non-negativity property
\[
(\gamma+\chi)\int|\nabla\epsilon|^2
 +\Bigl(\frac{\gamma-\mu}{2}-2\chi a\Bigr)\int\nabla\epsilon\cdot\nabla\omega
 +a\mu\int|\nabla\omega|^2 \ge0
\]
for all $\nabla\epsilon\in \R^2$ and all  $\nabla\omega\in \R^2$,
is guaranteed by the following condition
\begin{equation*}
 \Bigl(\frac{\gamma-\mu}{2}-2\chi a\Bigr)^2-4(\gamma+\chi)a\mu\le0,
\end{equation*}
that can be rewritten in an equivalent way as
\begin{equation}\label{conda}
 16\chi^2a^2-8(\gamma\chi+\mu\chi+2\mu\gamma)a+(\gamma-\mu)^2\le0.
\end{equation}
The discriminant above is positive. The choice  \eqref{choice-a} corresponds to the minimizer of the above expression and ensures that the above relation holds. For this choice of $a$, applying the square reduction yields
\begin{align*}
  (\gamma+\chi)\int|\nabla\epsilon|^2
 &+\Bigl(\frac{\gamma-\mu}{2}-2\chi a\Bigr)\int\nabla\epsilon\cdot\nabla\omega
 +a\mu\int|\nabla\omega|^2\\
 &=(\gamma+\chi)\int|\nabla\epsilon|^2-\frac{\mu(\chi+\gamma)}{\chi}\int\nabla\epsilon\cdot\nabla\omega
 +\frac{\mu}{4\chi^2}(\gamma\chi+\mu\chi+2\mu\gamma)\int|\nabla\omega|^2\\
 &=(\gamma+\chi)\int\Bigl|\nabla\epsilon-\frac{\mu}{2\chi}\nabla\omega\Bigr|^2
 +\frac{\mu\gamma}{4\chi^2}(\mu+\chi)\int|\nabla\omega|^2.
 \end{align*}
This implies identity~\eqref{enEO22}.

Finally, relation \eqref{smallgammapositive} follows as in the proof of Corollary \ref{corgamma0}. This completes the proof.
\end{proof}
\begin{remark}
In the case $\gamma=0$, the only possible choice of $a$ satisfying \eqref{conda} is the choice given in \eqref{choice-a}. In that case, \eqref{enEO22} is the same as \eqref{identity}. We choose to give a different proof of this new identity in the case $\gamma=0$ because the formal proof of Lemma \ref{lidentity} can be adapted to show Proposition \ref{prop-deriv} while the proof of Proposition \ref{pident} cannot. Indeed, in the proof of Proposition \ref{pident} we have to integrate by parts terms of the form $\int\Delta\omega\epsilon$ and such a term does not have a sign. In such terms, we can neither pass to the limit $\delta\to0$ (as we did in the proof of Proposition \ref{prop-deriv}), nor can we find a good sign as we did in relation \eqref{example}.  
\end{remark}

\section{Energy decay}
\label{sec-energy}
In the previous section, we used a new identity and combined it with a monotonicity argument (see, e.g., \cite{GuterresNichePerusatoZingano2023,ZinganoJMFM}) to obtain fast decay for the angular velocity, namely $\|h(t)\|_{L^2}^2 = o(t^{-1})$, and also for the gradient of $u$, $\|\nabla u (t)\|^2_{L^2} = o(t^{-1})$. We will continue in this direction and prove two results. The first result shows that  $\|u(t)\|_{L^2} \to 0$ as $t \to \infty$, see Theorem \ref{Thm_little_o} below. The second result states that a rate of decay for $\|u(t)\|_{L^2}$ implies a corresponding rate of decay for  $\|h(t)\|_{L^2}$ and for $\|\nabla u (t)\|_{L^2}$, see Proposition \ref{consequence1} below. These two results will be required to prove Theorem \ref{th:decay} in Section \ref{sec:decay}, but they are significant on their own.

\begin{theorem}\label{Thm_little_o}
Let $\mu,\chi>0$ and $\gamma\ge0$. Let $(u_0,h_0)\in L^2_\sigma(\R^2)\times L^2(\R^2)$
and $(u,h)$ be a finite energy solution of \eqref{MP2D}
with initial data $(u_0,h_0)$.
Then
\[
\|u(t)\|_{L^2}\to 0
\quad\text{as $t\to \infty$}.
\]
\end{theorem}

\begin{proof}

Let $\eta>0$. By Corollary~\ref{corgamma0} and Proposition \ref{pident},
there exists $t_0\ge0$ such that, 
\begin{equation}\label{bdh}
\text{for a.e. $t\ge t_0$},\quad
\|\nabla u(t)\|_{L^2},\,\| h(t)\|_{L^2} \leq \eta\,t^{-\frac{1}{2}}.
\end{equation}
Moreover, for any $L^2$ function $f$, we have $\|e^{\tau\Delta}f \|_{L^2} \to 0$ as $\tau\to \infty$. 
So $t_0$ can be chosen so that we also have 
\[
\|e^{(\mu+\chi)t\Delta} u_0\|_{L^2} \leq \eta, \quad\,\,\forall t \geq  t_0.
\]
We apply the Duhamel formula in the equation of $u$, the first line of \eqref{sysNLproj}, and obtain that
\begin{equation*}
u(t)=e^{(\mu+\chi)\,t\,\Delta} u_0
-\int_{0}^t e^{(\mu+\chi)\,(t-\tau)\,\Delta}\P (u\cdot\nabla u)(\tau) \dd\tau
- 2\chi \int_{0}^t e^{(\mu+\chi)\,(t-\tau)\,\Delta} \nabla^\perp h(\tau) \dd\tau.
\end{equation*}
We take the $L^2$ norm above:
\begin{align*}
    \|u(t)\|_{L^2} 
&\leq \|e^{(\mu+\chi)\,t\,\Delta} u_0\|_{L^2} + 
    \int_{0}^t \|e^{(\mu+\chi)\,(t-\tau)\,\Delta}\P (u\cdot\nabla u)(\tau)\|_{L^2} \dd\tau   + 2\chi \int_{0}^t \|e^{(\mu+\chi)\,(t-\tau)\,\Delta} \nabla^\perp h(\tau)\|_{L^2} \dd\tau\\
& \leq    \eta + I_1 + I_2
\end{align*}
for all $t\geq t_0$.
Now, we will use standard bounds for the heat kernel to estimate $I_1$ and $I_2$. For $I_1$, we get, for $t\ge 2t_0$,
\begin{equation*}
\begin{split}
I_1:&=\int_{0}^{t} 
\|e^{(\mu+\chi)\,(t-\tau)\,\Delta} \P(u\cdot\nabla u)(\tau)\|_{L^2} \dd\tau\\
&\leq\int_{0}^{t} 
\|e^{(\mu+\chi)\,(t-\tau)\,\Delta} (u\cdot\nabla u)(\tau)\|_{L^2} \dd\tau\\
&\lesssim\int_{0}^{t} 
(t-\tau)^{-1/2}\|u(\tau)\|_{L^2}\|\nabla u(\tau)\|_{L^2} \dd\tau\\
&=\int_{0}^{t_0} 
(t-\tau)^{-1/2}\|u(\tau)\|_{L^2}\|\nabla u(\tau)\|_{L^2} \dd\tau
+\int_{t_0}^{t} 
(t-\tau)^{-1/2}\|u(\tau)\|_{L^2}\|\nabla u(\tau)\|_{L^2} \dd\tau\\
&\lesssim t^{-1/2}\int_{0}^{t_0} \|u(\tau)\|_{L^2}\|\nabla u(\tau)\|_{L^2} \dd\tau+\eta \|z_0\|_{L^2}\int_{0}^{t}  (t-\tau)^{-1/2}\tau^{-1/2} \dd\tau\\    
& \lesssim \eta
    \end{split}
\end{equation*}
if $t_0$ is sufficiently large. We used above \eqref{bdh} and the energy estimates \eqref{enereq}.

On the other hand, for $I_2$, we use again \eqref{bdh} and estimate, for $t\ge2t_0$,
\begin{equation*}
\begin{split}
I_2:&= 2\chi \int_{0}^t \|e^{(\mu+\chi)\,(t-\tau)\,\Delta} \nabla^\perp h(\tau)\|_{L^2} \dd\tau\\
&\lesssim       \int_{0}^{t} (t-\tau)^{-1/2} \|h(\tau)\|_{L^2} \dd\tau\\
&=      \int_{0}^{t_0} (t-\tau)^{-1/2} \|h(\tau)\|_{L^2} \dd\tau + \int_{t_0}^{t} (t-\tau)^{-1/2} \|h(\tau)\|_{L^2} \dd\tau\\
&\lesssim   t^{-1/2}    \int_{0}^{t_0} \|h(\tau)\|_{L^2} \dd\tau
        +\eta\int_{0}^t (t-\tau)^{-1/2} \tau^{-1/2} \dd\tau\\
     &   \lesssim \eta
\end{split}
\end{equation*}
for $t_0$ sufficiently large. Therefore, $\|u(t)\|_{L^2}\leq C\eta$ for $t$ sufficiently large and this completes the proof.
\end{proof}

We require now the following simple lemma. 
\begin{lemma}\label{easylemma}
Let $f$ be a function integrable on the interval $[0,1]$ and $A\subset[0,1]$ a negligible set. There exists a sequence $t_n\to0$, $t_n\not\in A$ for all $n$, such that $t_n f(t_n)\to0$.
\end{lemma}
\begin{proof}
For every $n\in\N^*$, there exists $t_n\in[0,\frac1n]\setminus A$ such that $|t_n f(t_n)|\leq\frac1n$. Indeed, if this was not the case, we would have $|tf(t)|>\frac1n$ on $[0,\frac1n]\setminus A$ so
\begin{equation*}
\int_0^{\frac1n}|f(t)|\dt=\int_{[0,\frac1n]\setminus A}|f(t)|\dt\geq\int_{[0,\frac1n]\setminus A} \frac{1}{nt}\dt=\int_0^{\frac1n} \frac{1}{nt}\dt=+\infty
\end{equation*}
which contradicts the integrability of $f$. The sequence $t_n$ verifies the conclusion of the lemma.
\end{proof}

The second result of this section is the following proposition.
\begin{proposition}\label{consequence1}
Let $(u_0,h_0)\in L^2_\sigma(\R^2)\times L^2(\R^2)$ and $\Gamma\ge0$.
If $(u,h)$ is a finite energy solution of~\eqref{MP2D}
with initial data $(u_0,h_0)$, and if $\|u(t)\|^2_{L^2} = O(t^{-\Gamma})$ as $t \to \infty$, then 
\[
\|h(t)\|^2_{L^2} +\|\nabla u(t)\|^2_{L^2} 
= O(t^{-\Gamma - 1})
\qquad\text{a.e. as $t\to+\infty$}.
\]
\end{proposition}
The notation $O(t^{-\Gamma - 1})$ a.e. as $t\to+\infty$ means that there exists a negligible set $A$ such that the bound $O(t^{-\Gamma - 1})$ holds true for $t\not\in A$.
\begin{proof}
We assume that $\gamma=0$. The case $\gamma>0$ can be proved similarly.

Let us first assume that $\|u(t)\|_{L^2}^2+\|h(t)\|_{L^2}^2 = O(t^{-\Gamma})$ a.e.. 
Let $\alpha>\Gamma$. The following modified energy equality holds for all $t\ge0$:
\begin{multline*}
t^\alpha\Bigl(\nl2{u(t)}^2+\nl2{h(t)}^2\Bigr)
+2\mu\int_0^t s^\alpha \nl2{\nabla u(s)}^2 \dd s
+2\gamma\int_0^t s^\alpha \nl2{\nabla h(s)}^2 \dd s\\
+2\chi\int_0^t s^\alpha \nl2{(\omega-2h)(s)}^2 \dd s
=\alpha\,\displaystyle\int_0^t s^{\alpha-1}\Bigl(\nl2{u(s)}^2+\nl2{h(s)}^2\Bigr) \dd s.
\end{multline*}
The above equality is obtained formally by multiplying the two equations in the system \eqref{MP2D} respectively by 
$t^\alpha u(x,t)$ and $t^\alpha\,h(x,t)$
and integrating in space and by parts in time on $[0,t]$, and is justified rigorously for finite energy solutions arguing as in Proposition~\ref{prop-energ}.
So, when $\|u(t)\|_{L^2}^2+\|h(t)\|_{L^2}^2 = O(t^{-\Gamma})$ a.e., we deduce in particular
\begin{equation*}
    \int_0^t s^\alpha \Big( \nl2{\omega(s)}^2 +
      \nl2{(\omega-2h)(s)}^2 \Bigr)\dd s\, \lesssim t^{\alpha-\Gamma}.
\end{equation*}
But, recalling the definition of $\theta=h-\frac{\chi+\mu}{2\chi}\omega$,
we get from the triangular inequality
\[
\|h(s)\|_{L^2}^2+\|\theta(s)\|_{L^2}^2\lesssim
 \nl2{\omega(s)}^2 +\nl2{(\omega-2h)(s)}^2
\]
So,
\begin{equation}\label{starf}
\int_0^t s^\alpha\Bigl(\mu \|h(s)\|_{L^2}^2+\chi\|\theta(s)\|_{L^2}^2
\Bigr)\dd s \lesssim t^{\alpha-\Gamma}.
\end{equation}

We now go back to the computations made to establish the
new enstrophy identity of Section~\ref{sec:val-ide}, with some slight modifications: we multiply identity~\eqref{eq-theta} by 
$\chi\, t^{\alpha+1}\theta(t)$ and the
equation for $h$ in~\eqref{MP2D} by $\mu \,t^{\alpha+1}h(t)$,
then integrate in space and by parts in time on~$[0,t]$.
We deduce as in the proof of Proposition \ref{prop-deriv}, after neglecting the non-negative terms that one obtains on the left-hand side, that, for finite energy solutions, the function
\begin{equation*}
\psi_\alpha(t)=t^{\alpha+1} \Bigl(\mu\|h(t)\|_{L^2}^2+\chi\|\theta(t)\|_{L^2}^2\Bigr)
-
(\alpha+1)\int_0^t s^\alpha
\Bigl( \mu\, \|h(s)\|_{L^2}^2+\chi \|\theta(s)\|_{L^2}^2\Bigr)\,ds
\end{equation*}
has a negative derivative in the sense of the distributions. So there exists a negligible set $A$ such that $\psi_\alpha$ is equal to a decreasing function on $\R_+\setminus A$. Let $t>0$, $t\not\in A$. We apply Lemma \ref{easylemma} to the function $t\mapsto \mu\|h(t)\|_{L^2}^2+\chi\|\theta(t)\|_{L^2}^2$ which is integrable thanks to Lemma \ref{psil1}, and find a sequence of times $t_n$ such that $t_n\to0$, $t_n\not\in A$ and  $t_n( \mu\|h(t_n)\|_{L^2}^2+\chi\|\theta(t_n)\|_{L^2}^2)\to0$. By monotonicity, we have that $\psi_\alpha(t)\leq \psi_\alpha(t_n)$ if $n$ is large enough. We clearly have that 
$t_n^{\alpha+1}( \mu\|h(t_n)\|_{L^2}^2+\chi\|\theta(t_n)\|_{L^2}^2)\to0$ so $\psi_\alpha(t_n)\to0$. Letting $n\to\infty$ therefore implies that $\psi_\alpha(t)\leq 0$. We proved that 
\begin{equation*}
t^{\alpha+1} \Bigl(\mu\|h(t)\|_{L^2}^2+\chi\|\theta(t)\|_{L^2}^2\Bigr)
\le 
(\alpha+1)\int_0^t s^\alpha
\Bigl( \mu\, \|h(s)\|_{L^2}^2+\chi \|\theta(s)\|_{L^2}^2\Bigr)\,ds,
\end{equation*}
for almost all $t\ge0$.
Relation \eqref{starf} implies that, for almost all $t\ge0$,
\[
t^{\alpha+1}\Bigl(\mu \|h(t)\|_{L^2}^2+\chi\|\theta(t)\|_{L^2}^2\Bigr)
\lesssim
t^{\alpha-\Gamma}.
\]
Hence, we proved that, for finite energy solutions, and for all
$\Gamma\ge0$, the following implication holds:
\[
\label{im:i}
\|u(t)\|_{L^2}^2+\|h(t)\|_{L^2}^2 = O(t^{-\Gamma})\;\Longrightarrow\;
\|h(t)\|_{L^2}^2+\|\theta(t)\|_{L^2}^2
=O(t^{-\Gamma-1})
\qquad\text{a.e. as $t\to+\infty$}.
\tag{I}
\]

Now, we claim that for $\Gamma\ge0$  we have the following implication:
\[
\label{im:ii}
\|u(t)\|_{L^2}^2 = O(t^{-\Gamma})\;\Longrightarrow\;
\|h(t)\|_{L^2}^2 = O(t^{-\Gamma})
\qquad\text{as $t\to+\infty$}.
\tag{II}
\]
To see this, we consider first the case $0 \le \Gamma \leq 1$. 
We already know that $t \|h(t)\|^2 \to 0$ by Corollary~\ref{corgamma0}.
Thus, in particular, $\|h(t)\|^2=O(t^{-\Gamma})$, and \eqref{im:ii} trivially holds true.
Proceeding to the case where $1 < \Gamma \leq 2$, we have the following: 
as $t\to+\infty$, if $\|u(t)\|_{L^2}^2 = O(t^{-\Gamma})$, then $\|u(t)\|_{L^2}^2=O(t^{-1})$. Since we always have that $\|h(t)\|_{L^2}^2=O(t^{-1})$, \eqref{im:i} implies that  
$\|h(t)\|_{L^2}^2=O(t^{-2})=O(t^{-\Gamma})$, as $t\to+\infty$.
So \eqref{im:ii} holds in the case $1<\Gamma\le 2$ too. The general case $\Gamma$
is proved by boot-strapping, in the same way.

Therefore, implication~\eqref{im:i} improves to 
\[
\forall\Gamma\ge0,\quad
\|u(t)\|_{L^2}^2 = O(t^{-\Gamma})\;\Longrightarrow\;
\|h(t)\|_{L^2}^2+\|\theta(t)\|_{L^2}^2
=O(t^{-\Gamma-1})
\qquad\text{a.e. as $t\to+\infty$}.
\]
But, recalling the definition of $\theta$ and applying the triangular inequality,
we see that the conditions
\[
\|h(t)\|_{L^2}^2+\|\theta(t)\|_{L^2}^2=O(t^{-\Gamma-1})
\]
and
\[
\|h(t)\|_{L^2}^2+\|\omega(t)\|_{L^2}^2=O(t^{-\Gamma-1})
\]
are equivalent. The conclusion follows recalling that $\|\omega(t)\|_{L^2}=\|\nabla u(t)\|_{L^2}$.
\end{proof}

\begin{remark}
\label{rem:conse1}
Clearly, the assertion of Proposition~\ref{consequence1} remains true, with a simpler proof, for the
solution $(\ul,\hl)$ of the linear system \eqref{sysL} below.
In particular, if $(u_0,h_0)\in L^2_\sigma(\R^2)\times L^2(\R^2)$ and
$\|\ul(t)\|^2_{L^2} = O(t^{-\Gamma})$, as $t \to \infty$ for some $\Gamma\ge0$, then we have 
\[
\|\hl(t)\|^2_{L^2}= O(t^{-\Gamma - 1})
\qquad\text{and}
\qquad
\|\zl(t)\|_{L^2}^2\lesssim (1+t)^{-\Gamma}.
\]

\end{remark}

\section{Linear evolution: explicit solutions and precise asymptotics}
\label{sect-linear}

In this section, we will focus solely on the linear part of \eqref{MP2D} which reads
\begin{equation}\label{sysL}
\left\{
\begin{aligned}
  \partial_t \ul&=(\mu+\chi)\Delta \ul-2\chi\nabla^\perp \hl \\
  \partial_t \hl &=\gamma\Delta \hl +2\chi\oml-4\chi \hl 
\end{aligned}
\right.
\end{equation}
where $\oml=\curl\ul$, and compute explicitly its solutions. The initial data above is $\ul\big|_{t=0}=u_0$ and $\hl\big|_{t=0}=h_0$. Note that the gradient term in  \eqref{MP2D}  can be ignored when taking the linear part; indeed, the three terms $\partial_t \ul$, $\Delta \ul$ and $\nabla^\perp \hl$ are all divergence-free.

We denote $\zl=(\ul,\hl)$, $z_0=(u_0,h_0)$ and  $e^{tM}$ the semigroup associated to \eqref{sysL}, so that $\zl(t)=e^{tM}z_0$. We will compute explicitly  $e^{tM}$ by applying the Fourier transform to \eqref{sysL}.

We start by taking the curl of the first equation in \eqref{sysL}. We obtain the following equivalent system:
\begin{equation*}
\left\{
\begin{aligned}
  \partial_t \oml&=(\mu+\chi)\Delta \oml-2\chi\Delta \hl \\
  \partial_t \hl &=\gamma\Delta \hl +2\chi\oml-4\chi \hl 
\end{aligned}
\right.
\end{equation*}
We apply the Fourier transform and obtain
\begin{equation}\label{linfour}
\partial_t
\begin{pmatrix}
  \widehat\oml\\\widehat \hl 
\end{pmatrix}
=
\begin{pmatrix}
-(\mu+\chi)|\xi|^2 \widehat\oml+2\chi|\xi|^2\widehat \hl \\
-\gamma|\xi|^2 \widehat \hl +2\chi\widehat \oml-4\chi \widehat \hl   
\end{pmatrix}
=A\begin{pmatrix}
  \widehat\oml\\\widehat \hl 
\end{pmatrix}
\end{equation}
where $\xi$ denotes the variable in the Fourier space and the matrix $A$ is given by
\begin{equation*}
A=\begin{pmatrix}
-(\mu+\chi)|\xi|^2 &  2\chi|\xi|^2\\
2\chi &  -(\gamma|\xi|^2 +4\chi)
\end{pmatrix}.
\end{equation*}
The solution of \eqref{linfour} is given by
\begin{equation*}
\begin{pmatrix}
  \widehat\oml(\xi,t)\\\widehat \hl (\xi,t)
\end{pmatrix}
=e^{tA}
\begin{pmatrix}
  \widehat\oml_0(\xi)\\\widehat h_0(\xi)
\end{pmatrix}
 .
\end{equation*}
By the Biot-Savart law we know that
\begin{equation}\label{BS}
\widehat \ul(\xi,t)=-i\frac{\xi^\perp}{|\xi|^2}\widehat\oml(\xi,t).
\end{equation}

We have to compute now the exponential of the matrix $tA$. The formula for the exponential of a 2x2 matrix is well known. Let us introduce some notation:
\begin{equation}\label{notL}
\left\{
\begin{aligned}
\alpha&=\frac12(\mu+\chi+\gamma)R+2\chi\\
\beta&=\frac12(\mu+\chi-\gamma)R-2\chi\\
 D &=\beta^2+4\chi^2R\\
R&=|\xi|^2.
\end{aligned}
\right.
\end{equation}  

\begin{lemma}
We have that
\begin{equation*}
e^{tA}=\frac1{2\sqrt D}
\begin{pmatrix}
 e^{-t(\alpha-\sqrt D)}(\sqrt D-\beta)+e^{-t(\alpha+\sqrt D)}(\sqrt D+\beta)&
\hskip 1cm 2\chi R \big(e^{-t(\alpha-\sqrt D)}-e^{-t(\alpha+\sqrt D)}\big)\\
\hskip -2 cm 2\chi \big(e^{-t(\alpha-\sqrt D)}-e^{-t(\alpha+\sqrt D)}\big)&
\hskip -1.5cm e^{-t(\alpha-\sqrt D)}(\sqrt D+\beta)+e^{-t(\alpha+\sqrt D)}(\sqrt D-\beta)
\end{pmatrix}
  .
\end{equation*}
\end{lemma}
\begin{proof}
With the notation introduced above, we have that
\begin{equation*}
A=
\begin{pmatrix}
  -\alpha-\beta&2\chi R\\
2\chi&-\alpha+\beta
\end{pmatrix}.
\end{equation*}
The characteristic polynomial of $A$ is $\lambda^2-\lambda\tr(A)+\det(A)$. The discriminant of the characteristic polynomial is
$[\tr(A)]^2-4\det(A)=4\beta^2+16\chi^2 R$ and its roots, the eigenvalues of $A$, are $\lambda_\pm=-\alpha\pm\sqrt D$. Since the eigenvalues are real and distinct, because $\chi>0$, we can apply Putzer’s spectral formula, see \cite{putzer_avoiding_1966}, which for a 2x2 matrix reads
\begin{equation*}
e^A=e^{\lambda_+} I+\frac{e^{\lambda_+}-e^{\lambda_-}}{\lambda_+-\lambda_-}(A-\lambda_+I)=\frac{e^{\lambda_+}}{\lambda_+-\lambda_-}(A-\lambda_-I)
-\frac{e^{\lambda_-}}{\lambda_+-\lambda_-}(A-\lambda_+I).
\end{equation*}
Since $\lambda_+-\lambda_-=2\sqrt D$, applying the above formula for $e^{tA}$ yields
\begin{align*}
e^{tA}
&=\frac{e^{t\lambda_+}}{\lambda_+-\lambda_-}(A-\lambda_-I)
-\frac{e^{t\lambda_-}}{\lambda_+-\lambda_-}(A-\lambda_+I)\\
&=\frac{e^{-t(\alpha-\sqrt D)}}{2\sqrt D}[A-(-\alpha-\sqrt D)I]
-\frac{e^{-t(\alpha+\sqrt D)}}{2\sqrt D}[A-(-\alpha+\sqrt D)I]\\
&=\frac{e^{-t(\alpha-\sqrt D)}}{2\sqrt D}
  \begin{pmatrix}
    \sqrt D-\beta&2\chi R\\2\chi&\sqrt D+\beta
  \end{pmatrix}
-\frac{e^{-t(\alpha+\sqrt D)}}{2\sqrt D}
  \begin{pmatrix}
    -\beta-\sqrt D&2\chi R\\2\chi&\beta-\sqrt D
  \end{pmatrix}
  .
\end{align*}
\end{proof}

The previous lemma and the Biot-Savart law \eqref{BS} imply the following formula for $\widehat \ul$ and $\widehat \hl $.
\begin{proposition}\label{plf}
The Fourier transform of the solution of \eqref{sysL} is given by the following formula:
\begin{equation}
\label{solou2}
\left\{\begin{aligned}
\widehat \ul(\xi,t)&= E_{1,1}(\xi,t)\widehat u_0(\xi)-i\xi^\perp E_{1,2}(\xi,t)\widehat h_0(\xi)\\
\widehat \hl (\xi,t)&= i E_{2,1}(\xi,t)\xi^\perp\cdot\widehat u_0(\xi)+E_{2,2}(\xi,t)\widehat h_0(\xi)\\
\end{aligned}
\right.
\end{equation}
where
\begin{align}
E_{1,1}&=\frac{1}{2\sqrt{ D }}\bigl[e^{-t(\alpha-\sqrt D)}(\sqrt D-\beta)+e^{-t(\alpha+\sqrt D)}(\sqrt D+\beta)\bigr]\label{e11}\\
E_{1,2}=E_{2,1}&=\frac{\chi}{\sqrt D }\big(e^{-(\alpha-\sqrt{ D })t}-e^{-(\alpha+\sqrt{ D })t}\big)\label{e21}\\
E_{2,2}&=\frac{1}{2\sqrt{ D }}\bigl[e^{-(\alpha-\sqrt{ D })t}(\beta+\sqrt{ D })+e^{-(\alpha+\sqrt{ D })t}(\sqrt{ D }-\beta)\bigr].\label{e22}
\end{align}
\end{proposition}

Let $K$ denote the following 3x3 matrix with complex coefficients:
\begin{equation}
\label{symbK}
K(\xi,t)=
\begin{pmatrix}
  E_{1,1}(\xi,t)& 0& i\xi_2 E_{1,2}(\xi,t)\\
0&  E_{1,1}(\xi,t)& -i\xi_1 E_{1,2}(\xi,t)\\
- i\xi_2 E_{2,1}(\xi,t)& i\xi_1 E_{2,1}(\xi,t)&E_{2,2}(\xi,t)
\end{pmatrix}
.
\end{equation}

Proposition \ref{plf} says that the matrix $K$ is the symbol of the operator $e^{tM}$:
\begin{equation*}
\zl(t)=e^{tM}z_0=\mathscr{F}^{-1}(K(\xi,t)\widehat{z_0}(\xi)).
\end{equation*}

Next, we use the explicit formulas for  $\widehat \ul$ and $\widehat \hl $ to derive precise asymptotics for $\ul$ and $\hl $.

The quantities $\mu$, $\chi$ and $\gamma$ being fixed, we consider $\alpha=\alpha(R)$, $\beta=\beta(R)$ and $ D = D (R)$ as being functions of the variable $R$ defined on $\R_+$. Recalling that $\mu$, $\chi$ are all strictly positive and $\gamma \geq 0$, it is easy to check that for all $R>0$ we have that
\begin{equation*}
 D >0,\quad \sqrt D \pm\beta>0,\quad \alpha\pm\sqrt D >0.
\end{equation*}
In addition,
\begin{equation*}
\alpha+\sqrt D \stackrel{R\to0}\sim 4\chi,\qquad
\alpha+\sqrt D \stackrel{R\to\infty}\sim\frac R2(\mu+\chi+\gamma+|\mu+\chi-\gamma|)
\end{equation*}
and since
\begin{equation*}
\alpha-\sqrt D =\frac{\alpha^2- D }{\alpha+\sqrt D }
=\frac{\alpha^2-\beta^2-4\chi^2R}{\alpha+\sqrt D }
=\frac{4\chi\mu R+\gamma(\chi+\mu)R^2}{\alpha+\sqrt D }
\end{equation*}
we also have that 
\begin{equation*}
\alpha-\sqrt D \stackrel{R\to0}\sim \mu R,
\end{equation*}
\begin{equation*}
\alpha-\sqrt D \stackrel{R\to\infty}\sim R\frac{2\gamma(\chi+\mu)}{\mu+\chi+\gamma+|\mu+\chi-\gamma|}\quad\text{if }\gamma>0
\end{equation*}
and
\begin{equation*}
\alpha-\sqrt D \stackrel{R\to\infty}\sim \frac{4\chi\mu}{\mu+\chi}\quad\text{if }\gamma=0.
\end{equation*}
We infer that there exists two constants $C_1=C_1(\mu,\chi,\gamma)$ and  $C_2=C_2(\mu,\chi,\gamma)$ such that, for all $R\geq0$, 
\begin{equation}\label{firstbound1}
C_1R\leq\alpha-\sqrt D \leq C_2R\quad\text{if }\gamma>0,
\end{equation}
\begin{equation}\label{firstbound2}
C_1\frac{R}{1+R}\leq\alpha-\sqrt D \leq C_2\frac{R}{1+R}\quad\text{if }\gamma=0
\end{equation}
and
\begin{equation}\label{firstbound3}
C_1(1+R)\leq\alpha+\sqrt D \leq C_2(1+R).
\end{equation}
Note that if $R\leq1$ then relation \eqref{firstbound1} holds also when $\gamma=0$ too. Moreover, we  have for all $\gamma\geq0$
\begin{equation}
\label{C3}
\alpha-\sqrt D\geq C_3\quad\text{for all }R>1
\end{equation}
for some constant $C_3=C_3(\mu,\chi)>0$.

Before obtaining the asymptotic behavior of the symbol of $e^{tM}$, let us first prove some upper bounds for this symbol and, as a consequence, a couple of semigroup estimates.
\begin{proposition}\label{sim-heat}
There exist two positive constants $c=c(\mu,\chi,\gamma)$ and $C=C(\mu,\chi,\gamma)$ such that:
\begin{enumerate}[label=\alph*)]
\item If $\gamma>0$ then $|K(\xi,t)|\leq C  e^{-ct|\xi|^2}$ for all $t\geq0$ and $\xi\in\R^2$. In addition, we have the $L^1-L^2$ estimate $\|\nabla^k e^{tM}z_0\|_{L^2}\leq C  t^{-\frac{1+k}2}\|z_0\|_{L^1}$.
\item If $\gamma\geq0$ then  $|K(\xi,t)|\leq C  e^{-ct\min(1,|\xi|^2)}$ for all $t\geq0$ and $\xi\in\R^2$. Moreover, we have the $L^2-L^2$ estimate $\|e^{tM}z_0\|_{L^2}\leq C  \|z_0\|_{L^2}$.
\end{enumerate}
\end{proposition}
\begin{proof}
Observe first from \eqref{notL} that $|\beta|\leq \sqrt D$ so $|\sqrt D\pm\beta|\leq 2\sqrt D$. Also $\sqrt D\geq 2\chi|\xi|$. From relations \eqref{e11}-\eqref{symbK} we deduce the following estimate for the kernel $K$:
\begin{equation*}
|K(\xi,t)|\leq C (|E_{1,1}(\xi,t)|+|E_{2,2}(\xi,t)|+|\xi||E_{1,2}(\xi,t)|)\leq C\big(e^{-(\alpha-\sqrt{ D })t}+e^{-(\alpha+\sqrt{ D })t}\big)\leq Ce^{-(\alpha-\sqrt{ D })t}.
\end{equation*}
Relations \eqref{firstbound1}-\eqref{firstbound2} imply the announced upper bounds for $K(\xi,t)$.  These bounds for $K$ and the Plancherel theorem easily imply the  $L^1-L^2$ and $L^2-L^2$ semigroup estimates:
\begin{equation*}
\|\nabla^k e^{tM}z_0\|_{L^2}=\frac1{2\pi}\||\xi|^k K(\xi,t)\widehat{z_0}(\xi)\|_{L^2}\leq C  \||\xi|^ke^{-ct|\xi|^2}\|_{L^2}\|\widehat{z_0}\|_{L^\infty}\leq C  t^{-\frac{1+k}2}\|z_0\|_{L^1}
\end{equation*}
if $\gamma>0$, and
\begin{equation*}
\|e^{tM}z_0\|_{L^2}=\frac1{2\pi}\|K(\xi,t)\widehat{z_0}(\xi)\|_{L^2}\leq C  \sup_\xi e^{-ct\min(1,|\xi|^2)} \|\widehat{z_0}\|_{L^2}\leq C  \|z_0\|_{L^2}
\end{equation*}
if $\gamma=0$. This completes the proof.
\end{proof}
\begin{remark}
It is clear from the proof that, in the case $\gamma>0$, we also have the slightly different estimate
\begin{equation}\label{pl1}
\Big\|\nabla^k e^{tM}{\footnotesize \Big(\begin{matrix}
\P \mathscr{v}\\ \mathscr{h}
\end{matrix}\Big)}\Big\|_{L^2}\leq C  t^{-\frac{1+k}2}(\|\mathscr{v}\|_{L^1}+\|\mathscr{h}\|_{L^1}).
\end{equation}
\end{remark}
More general $L^p$-$L^q$ estimates are available for $\gamma>0$
(essentially the same as for the heat semigroup), see \cite{Chen-Price},
but we will not need them here.

In the sequel, we will denote by $C$ a generic constant depending only on  $\mu$, $\chi$ and $\gamma$, whose value can change from one line to another. We will use several times the following equality:
\begin{equation}\label{sup}
\sup_{R\geq0}R^ke^{-t\nu R}=\Big(\frac{k}{e\nu t}\Big)^k.
\end{equation}
where $k\in\N^*$.

The following proposition gives the precise asymptotic behavior of $E_{1,1}$, $E_{2,1}$ and $E_{2,2}$ as $t\to\infty$.
\begin{proposition}\label{prop-e11}
There exist two positive constants $c=c(\mu,\chi,\gamma)$ and $C=C(\mu,\chi,\gamma)$ such that:
\begin{enumerate}[label=\alph*)]
\item if $\gamma>0$ then 
\begin{equation*}
|E_{1,1}(\xi,t)-e^{-\mu t|\xi|^2}|\leq \frac{C}{t}e^{-ct|\xi|^2}\qquad\forall t\geq 1,\ \forall\xi\in\R^2;
\end{equation*}
\item if $\gamma=0$ then 
\begin{equation*}
|E_{1,1}(\xi,t)-e^{-\mu t|\xi|^2}|\leq \frac{C}{t}e^{-ct|\xi|^2}+\frac{C}{1+|\xi|^2}e^{-ct}\qquad\forall t\geq 1,\ \forall\xi\in\R^2.
\end{equation*}
\end{enumerate}
\end{proposition}
\begin{proof}
We observe first that, due to the inequality $|\beta|\leq\sqrt D $, we have
\begin{equation}\label{bound2}
0\leq\frac{\pm\beta+\sqrt D }{2\sqrt D }\leq 1.
\end{equation}

The functions $\beta$ and $ D $ are $C^1$ in $R$ and $ D >0$ so the function
\begin{equation*}
\R_+\ni R\mapsto\frac{-\beta+\sqrt D }{2\sqrt D }
\end{equation*}
is $C^1$. One can check that its value in $R=0$ is 1 and that it is bounded as $R\to\infty$. We infer that there exists a constant $C=C(\mu,\chi,\gamma)$ such that
\begin{equation}\label{bound1}
\big|\frac{-\beta+\sqrt D }{2\sqrt D }-1\big|\leq CR
\end{equation}
for all $R\geq0$. 

Next, we use the Taylor expansion of the square root $\sqrt{1+s}=1+\frac{s}2+O(s^2)$ as $s\to0$ to write
\begin{equation}\label{exp}
\begin{aligned}
\alpha-\sqrt D 
&=\frac12(\mu+\chi+\gamma)R+2\chi-\sqrt{\big[\frac12(\mu+\chi-\gamma)R-2\chi\big]^2+4\chi^2R}\\
&=\frac12(\mu+\chi+\gamma)R+2\chi-2\chi\sqrt{1+R\frac{\chi+\gamma-\mu}{2\chi}+R^2\frac{(\mu+\chi-\gamma)^2}{16\chi^2}}\\
&=\frac12(\mu+\chi+\gamma)R+2\chi-2\chi(1+R\frac{\chi+\gamma-\mu}{4\chi}+O(R^2))\\
&=\mu R+O(R^2)
\end{aligned}
\end{equation}
as $R\to0$. Since $\alpha$ and $\sqrt D $ grow at most linearly at infinity, the relation above holds true for all $R\in\R_+$ and not only when $R\to0$.

We bound first the second term on the LHS of \eqref{e11}. Using \eqref{bound2} and \eqref{firstbound3} we can estimate
\begin{equation*}
\frac{\beta+\sqrt{ D }}{2\sqrt{ D }}e^{-(\alpha+\sqrt{ D })t}
\leq e^{-(\alpha+\sqrt{ D })t}
\leq e^{-C_1(1+R)t}
= e^{-C_1t}e^{-C_1tR}
\leq \frac Ct e^{-C_1 tR}.
\end{equation*}

We now go to the first term on the LHS of \eqref{e11}. The estimate is different, depending on $\gamma$ being 0 or not.

Assume first that $\gamma>0$. We consider two cases.

If $tR^2>1$ then
\begin{equation*}
\frac{-\beta+\sqrt{ D }}{2\sqrt{ D }}e^{-(\alpha-\sqrt{ D })t}
\leq  e^{-(\alpha-\sqrt{ D })t}
\leq  e^{-C_1Rt}
= e^{-\frac{C_1Rt}2}e^{-\frac{C_1Rt}2}
\leq e^{-\frac{C_1\sqrt t}2}  e^{-\frac{C_1Rt}2}
\leq \frac Ct e^{-\frac{C_1Rt}2}
\end{equation*}
where we used \eqref{firstbound1}.

If $tR^2\leq1$ we write
\begin{align*}
\frac{-\beta+\sqrt{ D }}{2\sqrt{ D }}e^{-(\alpha-\sqrt{ D })t}
&=\frac{-\beta+\sqrt{ D }}{2\sqrt{ D }}e^{-t\mu R+O(tR^2)}\\
&=\frac{-\beta+\sqrt{ D }}{2\sqrt{ D }}e^{-t\mu R}
+\frac{-\beta+\sqrt{ D }}{2\sqrt{ D }}e^{-t\mu R}(e^{O(tR^2)}-1)\\
&=e^{-t\mu R}+O(R)e^{-t\mu R}+O(tR^2)e^{-t\mu R}\\
&=e^{-t\mu R}+O(R)e^{-\frac{t\mu R}2}e^{-\frac{t\mu R}2}+O(tR^2)e^{-\frac{t\mu R}2}e^{-\frac{t\mu R}2}
\end{align*}
where we also used \eqref{bound1} and \eqref{bound2}. We deduce from \eqref{sup} that
\begin{equation*}
O(R)e^{-\frac{t\mu R}2}=O\big(\frac1t\big)\quad\text{and}\quad O(tR^2)e^{-\frac{t\mu R}2}=O\big(\frac1t\big)
\end{equation*}
so
\begin{equation*}
\frac{-\beta+\sqrt{ D }}{2\sqrt{ D }}e^{-(\alpha-\sqrt{ D })t}=e^{-t\mu R}+O\big(\frac1t\big)e^{-\frac{t\mu R}2}.
\end{equation*}

The conclusion follows after we recall that $R=|\xi|^2$.

We assume now that $\gamma=0$. In this case, \eqref{firstbound1} is no longer true and must be replaced by \eqref{firstbound2}. But if $R\leq 1$, then \eqref{firstbound1} holds true also when $\gamma=0$. We infer that if $R\leq1$ then the argument above goes through, so it suffices to assume $R>1$. Observe that $\beta\sim\frac{\mu+\chi}2R$ and $\sqrt D\sim\frac{\mu+\chi}2R$ as $R\to\infty$. We infer that
\begin{equation*}
\frac{-\beta+\sqrt{ D }}{2\sqrt{ D }}
=\frac{D-\beta^2}{2\sqrt D(\sqrt D+\beta)}
=\frac{4\chi^2R}{2\sqrt D(\sqrt D+\beta)}
\stackrel{R\to\infty}\sim\frac{4\chi^2}{(\mu+\chi)^2R}
\end{equation*}
so
\begin{equation*}
\frac{-\beta+\sqrt{ D }}{2\sqrt{ D }}\leq\frac CR\quad\text{for all } R>1.
\end{equation*}
Recalling \eqref{C3} we can  bound the first term on the LHS of \eqref{e11} as follows:
\begin{equation*}
\frac{-\beta+\sqrt{ D }}{2\sqrt{ D }}e^{-(\alpha-\sqrt{ D })t}
\leq \frac CR e^{-(\alpha-\sqrt{ D })t}
\leq \frac CR e^{-C_3t}.
\end{equation*}
This completes the proof.
\end{proof}

A similar analysis can be carried out to deduce the asymptotics of the term $E_{2,2}$.

\begin{proposition}\label{prop-e22}
There exist two positive constants $c=c(\mu,\chi,\gamma)$ and $C=C(\mu,\chi,\gamma)$ such that:
\begin{enumerate}[label=\alph*)]
\item if $\gamma>0$
\begin{equation*}
\big|E_{2,2}(\xi,t)-\frac{|\xi|^2}4e^{-\mu t|\xi|^2}\big|\leq \frac{C}{t^2}e^{-c t|\xi|^2}\qquad\forall t\geq 1,\ \forall\xi\in\R^2;
\end{equation*}
\item if $\gamma=0$
\begin{equation*}
\big|E_{2,2}(\xi,t)-\frac{|\xi|^2}4e^{-\mu t|\xi|^2}\big|\leq \frac{C}{t^2}e^{-c t|\xi|^2}+Ce^{-c t}\qquad\forall t\geq 1,\ \forall\xi\in\R^2.
\end{equation*}
\end{enumerate}

\end{proposition}
\begin{proof}
As in the previous proof, we can bound 
\begin{equation}\label{firsterm}
\frac{-\beta+\sqrt{ D }}{2\sqrt{ D }}e^{-(\alpha+\sqrt{ D })t}
\leq e^{-(\alpha+\sqrt{ D })t}
\leq \frac C{t^2} e^{-C_1 tR}.
\end{equation}

It remains to bound the term
\begin{equation}\label{term}
\frac{\beta+\sqrt{ D }}{2\sqrt{ D }}e^{-(\alpha-\sqrt{ D })t}.
\end{equation}

We start by writing
\begin{equation*}
\frac{\beta+\sqrt{ D }}{2\sqrt{ D }}
=\frac{ D -\beta^2}{(\sqrt{ D }-\beta)2\sqrt{ D }}
=\frac{4\chi^2R}{(\sqrt{ D }-\beta)2\sqrt{ D }}
\stackrel{R\to0}\sim\frac R{4}
\end{equation*}
where we used that at $R=0$, we have $ D =4\chi^2$ and $\beta=-2\chi$. Taking also into account that $\frac{\beta+\sqrt{ D }}{2\sqrt{ D }}$ is bounded by 1, we obtain that
\begin{equation}\label{bound3}
\frac{\beta+\sqrt{ D }}{2\sqrt{ D }}=\frac R4+O(R^2)\quad\text{for all }R\geq0.
\end{equation}

We now estimate the term in \eqref{term}. 

Assume first the case $\gamma>0$.

If $tR^2>1$ then we use \eqref{exp} to bound the term in \eqref{term} as follows:
\begin{equation*}
\frac{\beta+\sqrt{ D }}{2\sqrt{ D }}e^{-(\alpha-\sqrt{ D })t}
\leq e^{-(\alpha-\sqrt{ D })t}
\leq e^{-C_1Rt}
=e^{-\frac{C_1Rt}2}e^{-\frac{C_1Rt}2}
\leq e^{-C_1\sqrt t}e^{-\frac{C_1Rt}2}
\leq \frac C{t^2}e^{-\frac{C_1Rt}2}.
\end{equation*}
Assume now that $tR^2\leq1$. We use \eqref{bound3} and \eqref{exp}:
\begin{align*}
\frac{\beta+\sqrt{ D }}{2\sqrt{ D }}e^{-(\alpha-\sqrt{ D })t}
&=\frac{\beta+\sqrt{ D }}{2\sqrt{ D }}e^{-t\mu R+O(tR^2)}\\
&=\frac{\beta+\sqrt{ D }}{2\sqrt{ D }}e^{-t\mu R}
+\frac{\beta+\sqrt{ D }}{2\sqrt{ D }}e^{-t\mu R}(e^{O(tR^2)}-1)\\
&=\frac R4 e^{-t\mu R}+O(R^2)e^{-t\mu R}+\frac{\beta+\sqrt{ D }}{2\sqrt{ D }}e^{-t\mu R}O(tR^2).
\end{align*}
Since $R\leq1/\sqrt t\leq 1$, from \eqref{bound3} we infer that $\frac{\beta+\sqrt{ D }}{2\sqrt{ D }}=O(R)$. Using \eqref{sup} we get that
\begin{equation*}
O(R^2)e^{-t\mu R}=O(R^2)e^{-\frac{t\mu R}2}e^{-\frac{t\mu R}2}
\leq \frac C{t^2}e^{-\frac{t\mu R}2}
\end{equation*}
and, similarly,
\begin{equation*}
\frac{\beta+\sqrt{ D }}{2\sqrt{ D }}e^{-t\mu R}O(tR^2)
=e^{-t\mu R}O(tR^3)\leq \frac C{t^2}e^{-\frac{t\mu R}2}.
\end{equation*}

We finally obtain that
\begin{equation*}
\Big|\frac{\beta+\sqrt{ D }}{2\sqrt{ D }}e^{-(\alpha-\sqrt{ D })t}-\frac R4 e^{-t\mu R}\Big|\leq \frac C{t^2}e^{-\frac{t\mu R}2}
\end{equation*}
for all $t\geq1$ and $tR^2\leq1$. 

The conclusion follows after recalling \eqref{firsterm}, the definition of $E_{2,2}$ given in \eqref{e22} and the relation $R=|\xi|^2$.

It remains to bound the term displayed in \eqref{term} in the case $\gamma=0$. We can assume that $R>1$ since when $R\leq 1$ the argument above goes through as in the proof of the previous proposition (by noticing that if $R\leq1$ then \eqref{firstbound1} holds true also when $\gamma=0$). If $\gamma=0$ and $R>1$ we recall that $\frac{\beta+\sqrt{ D }}{2\sqrt{ D }}\leq1$ and that $\alpha-\sqrt D\geq C_3$. We use this to bound the term in \eqref{term} as follows:
\begin{equation*}
\frac{\beta+\sqrt{ D }}{2\sqrt{ D }}e^{-(\alpha-\sqrt{ D })t}
\leq e^{-(\alpha-\sqrt{ D })t}
\leq e^{-C_3t}.
\end{equation*}
\end{proof}

We finish with the asymptotics for $E_{2,1}$.

\begin{proposition}\label{prop-e21}
There exist two positive constants $c=c(\mu,\chi,\gamma)$ and $C=C(\mu,\chi,\gamma)$ such that:
\begin{enumerate}[label=\alph*)]
\item if $\gamma>0$
\begin{equation*}
\big|E_{2,1}(\xi,t)-\frac12e^{-\mu t|\xi|^2}\big|\leq \frac{C}{t}e^{-c t|\xi|^2}\qquad\forall t\geq 1,\ \forall\xi\in\R^2;
\end{equation*}
\item if $\gamma=0$
\begin{equation*}
\big|E_{2,1}(\xi,t)-\frac12e^{-\mu t|\xi|^2}\big|\leq \frac{C}{t}e^{-c t|\xi|^2}+\frac{C}{1+|\xi|^2}e^{-c t}\qquad\forall t\geq 1,\ \forall\xi\in\R^2.
\end{equation*}
\end{enumerate}
\end{proposition}
\begin{proof}
Recall that 
\begin{equation*}
E_{2,1}=\frac{\chi}{\sqrt D }\big(e^{-(\alpha-\sqrt{ D })t}-e^{-(\alpha+\sqrt{ D })t}\big).
\end{equation*}
One can check that
\begin{equation*}
\frac{\chi}{\sqrt D }\stackrel{R\to0}\sim\frac12\quad\text{and}\quad \frac{\chi}{\sqrt D }\stackrel{R\to\infty}\longrightarrow0
\end{equation*}
so
\begin{equation*}
\frac{\chi}{\sqrt D }=\frac12+O(R)
\end{equation*}
for all $R\geq0$. Moreover $\frac{\chi}{\sqrt D }$ is bounded on $\R_+$.

We consider first the case $\gamma>0$.

With the same arguments as before we can show that
\begin{equation*}
\frac{\chi}{\sqrt D }e^{-(\alpha-\sqrt{ D })t}
=\big(\frac12+O(R)\big)e^{-t\mu R+O(tR^2)}
=\frac12 e^{-t\mu R}+O\big(\frac1t\big)e^{-\frac{t\mu R}2}
\end{equation*}
and
\begin{equation*}
\frac{\chi}{\sqrt D }e^{-(\alpha+\sqrt{ D })t}\leq Ce^{-(\alpha+\sqrt{ D })t}
\leq \frac C{t^2} e^{-C_1 tR}.
\end{equation*}
The conclusion follows in the case $\gamma>0$.

We consider now the case $\gamma=0$. As in the two previous proofs, it suffices to assume $R>1$. We observe that 
\begin{equation*}
\frac{\chi}{\sqrt D }\stackrel{R\to\infty}\sim\frac{2\chi}{(\mu+\chi)R}
\end{equation*}
so
\begin{equation*}
\frac{\chi}{\sqrt D }\leq\frac CR\quad\text{for all }R>1.
\end{equation*}
Then, for all $R>1$, we can use \eqref{C3} to estimate
\begin{equation*}
\frac{\chi}{\sqrt D }e^{-(\alpha-\sqrt{ D })t}\leq \frac CR e^{-C_3t}.
\end{equation*}
The proof is completed.
\end{proof}

As an application of the previous propositions, we can now establish
Theorem~\ref{th:linear}, about the asymptotic profile of the solution
$(\ul,\hl )$ to the linear system in \eqref{sysL}. 
We can now prove the main result of this section.

\begin{proof}[Proof of Theorem~\ref{th:linear}]
To estimate $\ul$ we use the first line of \eqref{solou2} and the Plancherel theorem:
\begin{align*}
\big\|\ul-e^{\mu t\Delta}u_0+\frac12\nabla^\perp e^{\mu t\Delta}h_0\big\|_{L^2} 
&=C\big\|\widehat \ul-e^{-\mu t|\xi|^2}\widehat{u_0}+\frac12e^{-\mu t|\xi|^2}i\xi^\perp \widehat{h_0}\big\|_{L^2}\\
&=C\big\|(E_{1,1}-e^{-\mu t|\xi|^2})\widehat{u_0}-(E_{2,1}-\frac12e^{-\mu t|\xi|^2})i\xi^\perp \widehat{h_0}\big\|_{L^2}\\
&\leq C\sup_\xi|E_{1,1}-e^{-\mu t|\xi|^2}|\, \|\widehat{u_0}\|_{L^2}+\sup_\xi(|\xi||E_{2,1}-\frac12e^{-\mu t|\xi|^2}|)\|\widehat{h_0}\|_{L^2}.
\end{align*}

For both cases $\gamma=0$ and $\gamma>0$  we use Propositions \ref{prop-e11} and \ref{prop-e21} to bound
\begin{equation*}
C\sup_\xi|E_{1,1}-e^{-\mu t|\xi|^2}|
\leq \frac Ct \sup_\xi e^{-ct|\xi|^2} +\sup_\xi\frac{C}{1+|\xi|^2}e^{-ct}\leq \frac Ct 
\end{equation*}
and
\begin{equation*}
\sup_\xi(|\xi||E_{2,1}-\frac12e^{-\mu t|\xi|^2}|)
\leq \frac Ct \sup_\xi (|\xi|e^{-ct|\xi|^2})
+\sup_\xi \frac{C|\xi|}{1+|\xi|^2}e^{-c t}
\leq \frac C{t^{\frac32}}
\end{equation*}
so
\begin{equation*}
\big\|\ul-e^{\mu t\Delta}u_0+\frac12\nabla^\perp e^{\mu t\Delta}h_0\big\|_{L^2} \leq \frac Ct \|u_0\|_{L^2}+\frac C{t^{\frac32}} \|h_0\|_{L^2}.
\end{equation*}

Similarly, we use  the second line of \eqref{solou2} together with Propositions \ref{prop-e22} and \ref{prop-e21} to estimate
\begin{align*}
\big\|\hl -\frac12\curl e^{\mu t\Delta}u_0+&\frac14\Delta e^{\mu t\Delta}h_0\big\|_{L^2}  
=C\big\|\widehat \hl -\frac12e^{-\mu t|\xi|^2}i\xi^\perp\cdot \widehat{u_0}-\frac14|\xi|^2 e^{-\mu t|\xi|^2}\widehat{h_0}\big\|_{L^2} \\
&=C\big\|(E_{2,1}-\frac12e^{-\mu t|\xi|^2})i\xi^\perp\cdot  \widehat{u_0}+(E_{2,2}-\frac14|\xi|^2 e^{-\mu t|\xi|^2})\widehat{h_0}\|_{L^2}\\
&\leq \sup_\xi(|\xi||E_{2,1}-\frac12e^{-\mu t|\xi|^2}|)\, \|\widehat{u_0}\|_{L^2}+C\sup_\xi|E_{2,2}-\frac14|\xi|^2e^{-\mu t|\xi|^2}|\, \|\widehat{h_0}\|_{L^2}\\
&\leq \frac C{t^{\frac32}} \|\widehat{u_0}\|_{L^2}+\big[\frac C{t^2} \sup_\xi e^{-ct|\xi|^2}+Ce^{-c t}\big]\|\widehat{h_0}\|_{L^2}\\
&\leq \frac C{t^{\frac32}} \|u_0\|_{L^2}+\frac C{t^2} \|h_0\|_{L^2}.
\end{align*}
This completes the proof.
\end{proof}

\section{Proof of Theorem~\ref{th:decay}} 
\label{sec:decay}

This section aims to show Theorem~\ref{th:decay}. We first need to prove some technical results.

The following lemma is proved in \cite{Brando-Revista}. 
\begin{lemma}\label{Brando_Revista}
Let $g(t), y(t)$ and $\beta(t)$ be three functions defined on $[0,\infty)$ such that $g$ is continuous, $y \in L^\infty_{\text{loc}}(\mathbb{R}^+)$ and $\beta \in L^1_{\text{loc}}(\mathbb{R}^+)$. Assume that (after a suitable modification of the values of
$y(t)$ on a set of measure zero)
\[
y(t)+\int_{s}^{t} g(r)^{2} y(r) d r \leq y(s)+\int_{s}^{t} \beta(r) d r
\]
holds for $s = 0$, a.e $s>0$ and all $t \geq s$. Let also $e(t)=\exp \left(\int_{0}^{t} g(r)^{2} d r\right)$. Then, 
\[ 
y(t) e(t) \leq y(0)+\int_{0}^{t} e(r) \beta(r) d r \quad \text { for all } t \geq 0.
\]
\end{lemma}
\begin{proof}
See e.g. \cite{Brando-Revista}.
\end{proof}

Another ingredient of the proof of Theorem~\ref{th:decay}, 
will be an inequality for low frequencies. This type of inequality goes back to Schonbek's Fourier-Splitting method
and was used also in Wiegner's work \cite{Wie87}, for the Navier-Stokes equations. For our problem, however, we must derive this Wiegner-type inequality more subtly, as we lack dissipation in the microrotation velocity.
We overcome this additional difficulty by taking advantage of the presence of a damping term in the equation of $h$.
Precisely, we have the following preliminary result.  

\begin{lemma} 
    Let $(u,h)$ be a finite energy solution
to~\eqref{MP2D}, with initial data $(u_0,h_0)\in L^2_\sigma(\R^2)\times L^2(\R^2)$. 
There exists $\delta=\delta(\chi,\mu)>0$ such that, for any
non-negative and monotonically decreasing function $g(\cdot)\in C^0(\R^+)$,
satisfying $g(0)\le 1$,
the following inequality holds:
\begin{equation}\label{WI}
    \|z(t)\|^2_{L^2}\,e(t) \lesssim \|z_0\|_{L^2}^2 
    + \delta\int_0^t e(\tau)\,g(\tau)^2\,I_u(\tau) \dd\tau, \quad \,\forall\,\,0 \leq s \leq t,
\end{equation}     
where $I_u(t):= \int_{|\xi| \le g(t)} |\widehat{u}(\xi,t)|^2 d\xi $ and 
$ e(t)=\exp\Bigl(\int_{0}^t \delta\, g(\tau)^2\dd \tau \Bigr)$.
\end{lemma}

\begin{proof}
First, let us observe that the energy equality~\eqref{enereq}
implies, for all $0\le s\le t$,
\[
\|z(t)\|_{L^2}^2+2(\mu+\chi)\int_s^t\|\nabla u\|_{L^2}^2+8\chi\int_s^t\|h\|_{L^2}^2
-8\chi\int_s^t\|\nabla u\|_{L^2}\|h\|_{L^2}\le 
\|z(s)\|_{L^2}^2.
\]
Notice that, in the above form, the inequality holds with or without microrotational dissipation, i.e., for any $\gamma\ge0$.
This implies, the inequality valid for all $\eta>0$, but mainly useful when
$1/2<\eta<\frac{\mu+\chi}{2\chi}$,
\[
\|z(t)\|_{L^2}^2+2(\mu+\chi-2\chi\eta)\int_s^t\|\nabla u\|_{L^2}^2+(8\chi-4\chi/\eta)\int_s^t\|h(s)\|_{L^2}^2
\le\|z(s)\|_{L^2}^2.
\]
Let $\delta,\eta>0$ be such that
\[
\delta=2(\mu+\chi-2\chi\eta)= 8\chi-4\chi/\eta.\\
\]
This is indeed possible by solving the quadratic equation in $\eta$,
then choosing
\begin{equation}
\label{deltaf}
\delta=\mu+5\chi-\sqrt{\mu^2+25\chi^2-6\chi\mu}.
\end{equation}
This leads to
\begin{equation*}
    \|z(t)\|^2_{L^2}+ \delta\int_s^t \|\nabla u (\tau) \|^2_{L^2} \dd \tau +\delta\int_s^t \|h (\tau) \|^2_{L^2} \dd \tau \le \|z(s)\|^2_{L^2},  
    \qquad\text{for all $0\le s\le t$.}
\end{equation*}		
However, 
\begin{equation}
\label{FS-meth}
\|\nabla u(t)\|^2_{L^2}
=\frac1{4\pi}\int |\xi|^2|\widehat u(\xi,t)|^2\dd\xi
\ge \frac{g(t)^2}{4\pi}\int_{|\xi|\ge g(t)}|\widehat u(\xi,t)|^2\dd\xi
=g(t)^2\|u(t)\|_{L^2}^2 - \frac{g(t)^2}{4\pi}I_u(t).
\end{equation}
Hence, using that $g(\cdot)\le1$,
\begin{equation}\label{62bis}
        \|z(t)\|^2_{L^2} + \delta\int_s^t\,g(\tau)^2 \|z (\tau) \|^2_{L^2} \dd \tau \le \|z(s)\|^2_{L^2} + 
        \frac{\delta}{4\pi} \int_s^t g(\tau)^2 I_u(\tau) \dd \tau, \quad \forall\,0 \le s \le t.
\end{equation}
By the definition of finite energy solutions, see Definition \ref{fesol},
$z \in L^\infty(\mathbb{R}^+, L^2 (\mathbb{R}^2) )$, so we can apply the previous Lemma \ref{Brando_Revista} and estimate~\eqref{WI} follows.
\end{proof}

Before delving into the proof of Theorem \ref{th:decay}, allow us to make some remarks concerning low-frequency estimates. These will be necessary for proving Theorem \ref{th:decay}.

Let us recall, see \eqref{symbK}, that we denoted by $K(\cdot,t)$ the symbol of the linear operator 
$z_0\mapsto \zl(\cdot,t)$,
i.e, for a.e. $\xi\in\R^2$,
\[
\widehat{\zl}(\xi,t)=K(\xi,t)\widehat z_0(\xi).
\]
Proposition~\ref{sim-heat} implies that $K(\cdot,t)\in L^\infty(\R^2)$.
Moreover, for finite energy solutions to~\eqref{MP2D}, we have
$|\widehat{u\otimes u}(\xi,t)|\le \|u(t)\|_{L^2}^2
\le \|z(t)\|_{L^2}^2$
and $|\widehat{u\otimes h}(\xi,t)|
\le \|u(t)\|_{L^2}\,\|h\|_{L^2}\le \|z(t)\|_{L^2}^2$.
This allows us to compute the Fourier transform term-by-term in the integral formulation of~\eqref{MP2D}, and thus obtain, for a.e. $\xi\in \R^2$,
\begin{equation*}
\begin{split}
\widehat z(\xi,t)
 &= K(\xi,t)\widehat z_0(\xi)
 	+\int_0^t K(\xi,t-s)\mathcal{F}
  	\begin{pmatrix}
	-\P\nabla\cdot(u\otimes u)\\
	-\nabla\cdot(u\otimes h)
	\end{pmatrix}(\xi,s)
	\dd s.	
\end{split}
\end{equation*}
The operators $\P\nabla\cdot$ and $\nabla\cdot$ being Fourier multipliers with symbol $\lesssim |\xi|$,
we deduce the important pointwise estimate
\begin{equation*}
\begin{split}
|\widehat z(\xi,t)|
 &\lesssim |\widehat \zl(\xi,t)|
 	+|\xi| \int_0^t \|z(s)\|_{L^2}^2\dd s,
 	\qquad \text{for a.e. $\xi\in \R^2$}.
\end{split}
\end{equation*}
Integrating $|\widehat z(\xi,t)|^2$ in the ball $\{|\xi|\le g(t)\}$, we get
\begin{equation}
\label{iteration0c}
I_u(t)\,\leq\,I_z(t):=\int_{|\xi| \le g(t)} |\widehat{z}(\xi,t)|^2 d\xi\,\lesssim \|\zl(t)\|_{L^2}^2
+g(t)^{4}\Bigl(\int_0^t \|z(s)\|_{L^2}^2\dd s\Bigr)^2.
\end{equation}

We need a last Lemma before proving Theorem \ref{th:decay}.

\begin{lemma}\label{crucial-lemma} 
Let $0<\Gamma<1$, $u_0\in L^2_\sigma(\R^2)$ and $h_0\in L^2(\R^2)$, such that $\|\zl(t)\|^2_{L^2} \lesssim (1+t)^{-\Gamma}$.
Then,
for any finite energy solution of~\eqref{MP2D},
we have
$\|z(t)\|_{L^2}^2 \lesssim (1+t)^{-\Gamma}$.
\end{lemma}
\begin{proof}
Let $\delta$ be defined in \eqref{deltaf}. We choose $A=A(\chi,\mu)\ge10$ large enough, such that, with the choice
\[
g(t)^2 := \frac{k}{\delta}(t+A)^{-1}\,[\,\log(t+A)\,]^{-1},\qquad k=1,2,
\] 
the required condition $g(\cdot)\le1$ holds true.
Then, 
\[
e(t)=\exp{\int_0^t \delta\, g(s)^2\dd s } = \frac{[\,\log(t+A)\,]^k}{[\log A]^k}.
\]
With this choice of $g$, by using the inequality \eqref{iteration0c} we obtain
\begin{equation*}
I_z(t)\lesssim 
\|\zl(t)\|_{L^2}^2+(t+A)^{-2}\,[\log(t+A)]^{-2}\,\Bigl(\int_0^t \|z(s)\|^2_{L^2}\dd s\Bigr)^2.
\end{equation*}
By applying \eqref{WI}, we have, for all $t\ge0$ and $k=1,2$,
\begin{multline}\label{WI_log_k}
        \|z(t)\|^2_{L^2}\,[\log(t+A)]^k\,
        \lesssim
        \|z_0\|^2_{L^2}+\int_0^t [\log(s+A)]^{k-1}\,(s+A)^{-1}
        \,\bigg\{
        \|\zl(s)\|^2_{L^2} \\ +(s+A)^{-2}\,[\log(s+A)]^{-2}\,\Bigl(\int_0^s \|z(r)\|^2_{L^2}\dd r\Bigr)^2\bigg\} \dd s .
\end{multline}

Recall the hypothesis  $\|\zl(t)\|_{L^2}^2\lesssim (1+t)^{-\Gamma}$. We use now the fact that \(\|z(t)\|^2_{L^2}\le \|z_0\|^2_{L^2}\), see \eqref{enereq}, and choose \(k = 1\) in the inequality~\eqref{WI_log_k}. 
Thus, we obtain:
 \begin{equation*}
    \begin{split}
        \|z(t)\|^2_{L^2}\,\log(t+A)
        &\lesssim 
        1+\int_0^t \,(s+A)^{-1}\,\|\zl(s)\|^2_{L^2}\dd s
        +\int_0^s \,(s+A)^{-1} \,[\log(s+A)]^{-2}\dd s
        \\
        &\lesssim 
        1 + \int_0^t (s+A)^{-1} (1+s)^{-\Gamma} \dd s \\
        &\lesssim 1.
    \end{split}
\end{equation*}
This logarithmic decay for $\|z(t)\|_{L^2}^2$ implies
that 
$(\int_0^s \|z(r)\|_{L^2}^2\dd r)^2\lesssim (s+A)^2[\log(s+A)]^{-2}$.
Using the last inequality in \eqref{WI_log_k}, we obtain:
\begin{equation*}
        \|z(t)\|^2_{L^2}\,[\log(t+A)]^k
        \lesssim
        1
        +\int_0^t \,[\log(s+A)]^{k-1}\,\,(s+A)^{-1-\Gamma} \dd s
        + \int_0^t \,[\log(s+A)]^{k-5}\,(s+A)^{-1} \dd s\lesssim  1.
\end{equation*}
By taking $k = 2$ in the inequality above, we obtain the improved logarithmic decay given by $\|z(t)\|^2_{L^2} \lesssim [\log(t+A)]^{-2}$. So, $\int_0^s\,\|z(r)\|^2_{L^2} \dd r \lesssim (s+A)\,[\log(s+A)\,]^{-2}$. 
Now, we must use this fact, with a new choice for $g$, 
to improve the estimate for \(I_z\).
Recalling~\eqref{iteration0c}, we obtain in this way
\[
I_z(t) \lesssim \|\zl(t)\|^2_{L^2} + g(t)^4\,(t+A)\,[\log(t+A)\,]^{-2} \int_0^t\,\|z(s)\|^2_{L^2} \dd s.
\] 
Thus  \eqref{WI} takes the following form:
\begin{equation}\label{WI_prepared_0} 
        \|z(t)\|^2_{L^2}\,e(t)
        \lesssim 
        1+\int_{0}^t e(s)\,g(s)^2\bigg\{
        \|\zl(s)\|^2_{L^2}  + g(s)^4 (s+A)[\log(s+A)]^{-2}\int_0^s \|z(r)\|^2_{L^2}\dd r\bigg\} \dd s .
\end{equation}
But, according to Lemma's assumptions, we have
$\|\zl(t)\|^2_{L^2} \lesssim (1+t)^{-\Gamma}$.
We now choose an arbitrary 
$\alpha$ such that $\Gamma<\alpha<1$ and we also choose $A\ge\max\{\alpha/\delta,1\}$.
Then, the splitting function $g(t)^2:=\frac{\alpha}{\delta} (t+A)^{-1}$
satisfies the required conditions for the validity of estimate~\eqref{WI}.
With this choice of $g$, we obtain 
$ e(t) =\frac{1}{A^\alpha}(t+A)^\alpha$.
Then, by \eqref{WI_prepared_0},
\begin{equation}\label{WI_prepared_1} 
    \begin{split}
        \|z(t)\|^2_{L^2}\,(t+A)^\alpha
        &\lesssim  1 +\int_0^t (s+A)^{\alpha-\Gamma-1}\dd s  + \int_0^t (s+A)^{\alpha-2}[\log(s+A)]^{-2} \int_0^s \|z(r)\|^2_{L^2}\dd r \dd s \\
        &\lesssim (t+A)^{\alpha-\Gamma} + \int_0^t (s+A)^{\alpha-2}[\log(s+A)]^{-2} \int_0^s \|z(r)\|^2_{L^2} \dd r \dd s .
    \end{split}
\end{equation}
Let us now estimate the term $ \int_0^s \|z(r)\|^2_{L^2}\dd r$. First, we define two continuous functions, namely, $y(t)$ and $\mathbb{Y}(t)$, for $t \geq 1$, as follows:  
\[ 
y(t) := \int_{t-1}^t \|z(r)\|^2_{L^2} (r+A)^\alpha \dd r
\quad\text{and}\quad
\mathbb{Y}(t):= \max\{y(\tau);\,1\leq\tau\leq t\}.
\]
Now, denoting by $[s]$ the integer part of $s$, we observe that,
for $s\ge1$
\begin{equation*}
    \begin{split}
         \int_0^s \|z(r)\|^2_{L^2}\dd r 
        &\lesssim 1 + \sum_{j=0}^{[s]-1} \int_{s-j-1}^{s-j} \|z(r)\|^2_{L^2} \dd r \\ 
        &\lesssim 1 + 
        \sum_{j=0}^{[s]-1}(s-j)^{-\alpha} \int_{s-j-1}^{s-j} \|z(r)\|^2_{L^2} (r+A)^\alpha \dd r \\ 
        &\lesssim 1 + \sum_{j=0}^{[s]-1}(s-j)^{-\alpha} y(s-j) \\ 
        &\lesssim 1 + \mathbb{Y}(s)\sum_{j=0}^{[s]-1}(s-j)^{-\alpha}  
        \\ 
        &\lesssim 1 + \mathbb{Y}(s) (s+A)^{1-\alpha}.
    \end{split} 
\end{equation*}
In the last inequality we used that $\alpha<1$.
Plugging this estimate into \eqref{WI_prepared_1}
we deduce, for all $t\ge1$,
\begin{equation}\label{67}
    \begin{split}
        \|z(t)\|^2_{L^2}\,(t+A)^\alpha
        &\lesssim (t+A)^{\alpha-\Gamma}
        + \int_1^t (s+A)^{-1}[\log(s+A)]^{-2}\mathbb{Y}(s)\dd s .
    \end{split}
\end{equation}
Integrating the latter inequality over the interval $[t-1,t]$, we obtain
\begin{equation*}
    \begin{split}
        y(t) \lesssim (t+A)^{\alpha-\Gamma} + \int_1^t (s+A)^{-1}\,[\log(s+A)]^{-2}\, \mathbb{Y}(s) \dd s, \quad t\geq 1 
     \end{split}
\end{equation*}
Hence, for some constant $C>0$
\begin{equation*}
        \mathbb{Y}(t) \leq C(t+A)^{\alpha-\Gamma} + C\int_1^t (s+A)^{-1}\,[\log(s+A)]^{-2}\, \mathbb{Y}(s) \dd s, \quad t\geq 1.
\end{equation*}
The classical Grönwall's lemma implies, for $t\ge1$,
\begin{equation*}
        \mathbb{Y}(t) \leq C(t+A)^{\alpha-\Gamma} 
        +\int_1^t C^2(s+A)^{\alpha-\Gamma-1}\,[\log(s+A)]^{-2}
        \exp\Bigl(\int_s^t C(\tau+A)^{-1}[\log(\tau+A)]^{-2}\dd\tau\Bigr) \dd s.
\end{equation*}
As the map $\tau\mapsto (\tau+A)^{-1}[\log(\tau+A)]^{-2}$ is integrable on $[1,+\infty[$,
we obtain, for another constant $C'>0$, 
\begin{equation*}
    \begin{split}
        \mathbb{Y}(t) \leq C(t+A)^{\alpha-\Gamma} 
        + C'\int_1^t (s+A)^{\alpha-\Gamma-1}\,[\log(s+A)]^{-2} \dd s
        \lesssim (t+A)^{\alpha-\Gamma}, \quad t\geq 1.
    \end{split}
\end{equation*}
Hence, from \eqref{67}, $\|z(t)\|^2_{L^2} \lesssim (t+A)^{-\Gamma}$, for all $t>0$,  which concludes the proof of the lemma.
\end{proof}

We now apply this Lemma to establish the first part of
Theorem~\ref{th:decay}.

\subsection*{Proof of Part i) of Theorem~\ref{th:decay}}
\

We begin by observing that, thanks to \eqref{ul11}, the hypothesis 
$\|e^{\mu t\Delta}(u_0-\frac12\nabla^\perp h_0)\|_{L^2}^2=O(t^{-\Gamma})$ as $t\to+\infty$ with $0 \le \Gamma\le 2$ is equivalent to  $\|u_L(t)\|_{L^2}^2\lesssim (1+t)^{-\Gamma}$. \label{startproof} Remark~\ref{rem:conse1} then implies that $\|\zl(t)\|^2_{L^2} = O(t^{-\Gamma})$ as $t\to+\infty$.

\begin{proof}[Proof of Part i) of Theorem~\ref{th:decay} in the case $0<\Gamma<1$.]\ \\
\indent Let $0<\Gamma<1$.
Lemma~\ref{crucial-lemma} applies, implying 
$\|u(t)\|_{L^2}^2+\|h(t)\|_{L^2}^2=O(t^{-\Gamma})$.
Proposition~\ref{consequence1}, then implies the improved decay
$\|h(t)\|_{L^2}^2=O(t^{-\Gamma-1})$ and the assertion i) follows in this case.
\end{proof}

\begin{proof}[Proof of Part i) of Theorem~\ref{th:decay} in the case $1\le\Gamma\leq 2$]\ \\
\indent Let us go back to~\eqref{WI_prepared_0}, and choose as before the splitting function $g(t)^2=\frac{\alpha}{\delta}(t+A)^{-1}$, but this time
with a larger exponent $\alpha$ (e.g., $\alpha=3$  will do) and $A\ge\max\{\alpha/\delta,1\}$.
Let us now go back to~\eqref{iteration0c},
and use the information 
$\|z(t)\|_{L^2}^2\lesssim (1+t)^{-1+\epsilon}$, 
valid for all $0<\epsilon<1$, because Assertion i)
of Theorem~\ref{th:decay} was already established for the decay
exponent $1-\epsilon$. We obtain
\begin{equation}
\label{itenew}
I_z(t)=\int_{|\xi| \le g(t)} |\widehat{z}(\xi,t)|^2 \dd\xi
\lesssim \|\zl\|_{L^2}^2
+g(t)^{4}\Bigl(\int_0^t \|z(s)\|_{L^2}^2\dd s\Bigr)^2\\
\lesssim
(t+A)^{-\Gamma}+(t+A)^{-2+2\epsilon}.
\end{equation}
Now, if $1\le \Gamma<2$, we can choose $\epsilon=1-\Gamma/2$
and obtain
\[
I_z(t)\lesssim (t+A)^{-\Gamma}.
\]
Plugging this in inequality~\eqref{WI} we obtain, observing that
$I_u\le I_z$,\begin{equation*}
\|z(t)\|_{L^2}^2 \,(t+A)^\alpha 
\lesssim 1+\int_0^t(s+A)^{\alpha-1-\Gamma}\dd s
\end{equation*}
and infer that $\|z(t)\|_{L^2}^2=O(t^{-\Gamma})$ as $t\to+\infty$.
In the case $\Gamma=2$, we can use e.g.,
the information $\|z(t)\|_{L^2}\lesssim(t+A)^{-3/2}$ to deduce 
$\int_0^{+\infty}\|z(s)\|_{L^2}^2\dd s\lesssim 1$ from the previous case
and improve estimate~\eqref{itenew}
to 
\[
I_z(t)\lesssim (t+A)^{-2}.
\]
Then one concludes as above that
$\|z(t)\|_{L^2}^2=O(t^{-\Gamma})$ as $t\to+\infty$
also in the case $\Gamma=2$.
The application of Proposition~\ref{consequence1} now yields
 $\|h(t)\|_{L^2}^2=O(t^{-\Gamma-1})$.
\end{proof}

\subsection*{Proof of Part ii) of Theorem~\ref{th:decay}}
\ 

Let us start with a useful lemma.
\begin{lemma}
\label{lem:degr}
Let $\Gamma\ge0$ and  $z_0=(u_0,h_0)\in L^2_\sigma\times L^2$ be such that
$\|\zl(t)\|^2=O(t^{-\Gamma})$ as $t\to+\infty$. 
When $\gamma=0$ we additionally assume that $\int (1+|\xi|)|\widehat z_0(\xi)|\dd\xi <\infty$. (No additional assumption is required when~$\gamma>0$).
Then we have
\begin{equation*}
\|\zl(t)\|_{L^\infty}^2+\|\nabla  \zl(t)\|_{L^\infty}^2
=O(t^{-1-\Gamma})
\qquad\text{as $t\to+\infty$}.
\end{equation*}
\end{lemma}
\begin{proof} 
Let us start with the case $\gamma=0$. Recall that the matrix $K$ defined in \eqref{symbK} is the symbol of $e^{tM}$, the semigroup associated to the linear system  \eqref{sysL}. Observe  by the semigroup property $e^{tM}=e^{\frac t2 M}e^{\frac t2 M}$  that
\begin{equation*}
\widehat{\zl}(t)=K(\xi,t)\widehat{z_0}=K(\xi,t/2)\widehat{\zl}(t/2).
\end{equation*}
Using the pointwise bound
of the symbol $K$ proved in Proposition \ref{sim-heat}, we obtain
\[
\begin{split}
\|\zl(t)\|_{L^\infty}&+\|\nabla \zl(t)\|_{L^\infty}
\le
 \|\widehat{\zl}(t)\|_{L^1}+ \|\widehat{\nabla \zl}(t)\|_{L^1}\\
&\leq\int_{|\xi|\le 1}\big(|K(\xi,t/2) \widehat{\zl}(t/2)|+|\xi|\,|K(\xi,t/2)\widehat{\zl}(t/2)|\big)
+\int_{|\xi|\ge1}\big(|K(\xi,t) \widehat{z_0}|+|\xi|\, |K(\xi,t)\widehat{z_0}|\big)
\\
&\leq \|K(\xi,t/2)\|_{L^2(|\xi\leq1)}\|\zl(t/2)\|_{L^2}
+\|K(\xi,t)\|_{L^\infty(|\xi\geq1)}\|(1+|\xi|)\widehat{z_0}\|_{L^1(|\xi|\geq1)}\\
&\lesssim\|e^{-\frac{c}{2}t|\xi|^2}\|_{L^2}\|\zl(t/2)\|_{L^2}+e^{-ct}\|(1+|\xi|)\widehat{z_0}\|_{L^1}\\
&\lesssim t^{-\frac12}t^{-\frac\Gamma2} +e^{-ct}\|(1+|\xi|)\widehat{z_0}\|_{L^1}\\
&\lesssim t^{-\frac12-\frac\Gamma2}.
\end{split}
\]

Let us now consider the case $\gamma>0$. In this case, as observed in Proposition \ref{sim-heat},
the Gaussian bound for the symbol $K(\cdot,t)$ is available for all $\xi$ and not only for $|\xi|\leq1$. So we can estimate as above without splitting in two pieces depending on $|\xi|\leq1$ or $|\xi|\geq1$:
\begin{equation*}
\|\zl(t)\|_{L^\infty}+\|\nabla \zl(t)\|_{L^\infty}
\leq \|(1+|\xi|)K(\xi,t/2)\|_{L^2}\|\zl(t/2)\|_{L^2}
\lesssim (t^{-\frac12}+t^{-1})t^{-\frac\Gamma2}
=O(t^{-1/2-\Gamma/2})
\end{equation*}
as $t\to+\infty$.
The conclusion of Lemma~\ref{lem:degr} follows.
\end{proof}

We continue with the proof of Part~ii) of Theorem~\ref{th:decay}.
Let us write an energy equality satisfied by the difference 
$z-\zl=(u-\ul,h-\hl)$ between a finite energy solution to~\eqref{MP2D}
and the corresponding linear system~\eqref{sysL}.
For the sake of the simplicity of the presentation, we proceed formally, but
the energy inequality~\eqref{EID} below can be established rigorously
with the argument of Section~\ref{sec:existence}.
The couple $(u-\ul,h-\hl)$ verifies
\begin{equation}
\label{sys:diff}
\left\{
\begin{aligned}
&\partial_t(u-\ul)+u\cdot \nabla u+\nabla p=(\mu+\chi)\Delta (u-\ul)
	-2\chi\nabla^\perp (h-\hl),\\
&\partial_t(h-\hl)+(u\cdot\nabla)h=\gamma\Delta (h-\hl)+ 2\chi\curl(u-\ul)-4\chi (h-\hl),\\
&\nabla\cdot (u-\bar u)=0.
\end{aligned}
\right.
\end{equation}

We argue now that we can make $L^2$ energy estimates above: multiply the first equation by $u-\ul$, the second equation by $h-\hl$, and integrate in space and time from $s$ to $t$. Doing that requires in fact to multiply the equation of $z$ by $z$, the equation of $z$ by $\zl$, the equation of $\zl$ by $z$ and the equation of $\zl$ by $\zl$. Multiplying the equation of $z$ by $z$ and integrating in space and time is justified; this is the content of Proposition~\ref{prop-energ}. The other three operations are also justified because $\zl$ is smooth. We conclude that we can multiply the equation of $z-\zl$ by $z-\zl$ and integrate in space and time. 

We estimate 
\begin{equation}
\label{estinte1}
\Bigl|\int (u\cdot\nabla u)\cdot(u-\ul)(t)\Bigr| 
\le \|\nabla(u-\ul)(t)\|_{L^2}\|u(t)\|_{L^2}\|\ul(t)\|_{L^\infty}
\end{equation}
and 
\begin{equation}
\label{estinte2}
\Bigl|\int (u\cdot\nabla h)\cdot(h-\hl)(t)\Bigr| 
\le \|(h-\hl)(t)\|_{L^2}
\|u(t)\|_{L^2}\|\nabla \hl(t)\|_{L^\infty}.
\end{equation}
To be precise, the integral above is not necessarily convergent because $u$ and $h$ are not regular enough to be able to say that $\iint  u\cdot\nabla h\cdot h=0$. But this possibly divergent integral can be neglected as a consequence of Proposition~\ref{prop-energ}.

Multiplying the equation for $z-\zl=(u-\ul,h-\hl)$ given in \eqref{sys:diff} by $z-\zl$, integrating in space and time from $s$ to $t$ and using \eqref{estinte1} and \eqref{estinte2} results in
\begin{multline*}
\nl2{(z-\zl)(t)}^2+2\mu\int_s^t\nl2{\nabla (u-\ul)}^2+2\gamma\int_s^t\nl2{\nabla (h-\hl)}^2
+2\chi\int_s^t\nl2{\curl (u-\ul)-2(h-\hl)}^2\\
\leq \nl2{(z-\zl)(s)}^2
+2\int_s^t\|\nabla(u-\ul)\|_{L^2}\|u\|_{L^2}\|\ul\|_{L^\infty}+2\int_s^t \|h-\hl\|_{L^2}
\|u\|_{L^2}\|\nabla \hl\|_{L^\infty}
\end{multline*}
for all $0< s\leq t<\infty$.

The term with $\gamma$ we ignore. The term with $\chi$ we write
\begin{align*}
2\chi\nl2{\curl (u-\ul)-&2(h-\hl)}^2\\
&=2\chi\nl2{\curl (u-\ul)}^2+8\chi\nl2{h-\hl}^2-8\chi\int \curl (u-\ul)(h-\hl)\\
&\geq -a \nl2{\curl (u-\ul)}^2+b\nl2{h-\hl}^2\\
&= -a \nl2{\nabla (u-\ul)}^2+b\nl2{h-\hl}^2
\end{align*}
where $a,b>0$ must verify $16\chi^2\leq(2\chi+a)(8\chi-b)$. One can choose such $a,b$ and assume in addition that $a<2\mu$. We get
\begin{align*}
\nl2{(z-&\zl)(t)}^2+(2\mu-a)\int_s^t\nl2{\nabla (u-\ul)}^2 +b\nl2{h-\hl}^2\\
&\leq \nl2{(z-\zl)(s)}^2
+2\int_s^t\|\nabla(u-\ul)\|_{L^2}\|u\|_{L^2}\|\ul\|_{L^\infty}+2\int_s^t \|h-\hl\|_{L^2}
\|u\|_{L^2}\|\nabla \hl\|_{L^\infty}
\end{align*}
Applying Young inequality, we can find $c_0=c_0(\chi,\mu)>0$
and $c_1(\chi,\mu,s_0)>0$ such that, for any $0< s_0\leq s \le t$
\begin{equation}
\label{EID}
\begin{split}
&\|(z-\zl)(t)\|_{L^2}^2- \|(z-\zl)(s)\|_{L^2}^2
+c_0\int_s^t\|\nabla (u-\ul)(r)\|_{L^2}^2\dd r
  + c_0\int_s^t\|(h-\hl)(t)\|_{L^2}^2\\ 
&\hskip3cm
\lesssim \int_s^t\|u(t)\|_{L^2}^2\Bigl(\|\ul(t)\|_{L^\infty}^2
+\|\nabla\hl\|_{L^\infty}^2\Bigr)\dd r\\
&\hskip3cm
\le c_1\int_s^t (1+r)^{-1-2\Gamma}\dd r,
\end{split}
\end{equation}
where in the last inequality we applied Lemma~\ref{lem:degr} and Part~i)
of Theorem~\ref{th:decay}. We excluded $s_0=0$ because of the 
possible singularity as $t\to0$ of $\|\ul(t)\|_{L^\infty}$ and $\|\nabla \hl(t)\|_{L^\infty}$.

Let us now choose a splitting function of the form
$g(t)^2=\frac{\alpha}{c_0}(t+A)^{-1}$, where $\alpha=3$ (any large enough 
constant $\alpha$ would do) and $A\ge1$ is such that $g(t)\le 1$ for all
$t\ge0$.
Let us set
\[
\begin{split}
I_{u-\ul}(t)&:=\int_{|\xi|\le g(t)} |\widehat{u}-\widehat{\ul}\,|^2(\xi,t)\dd\xi,\qquad\text{and}\\
I_{z-\zl}(t)&:=\int_{|\xi|\le g(t)} |\widehat{z}-\widehat{\zl}\,|^2(\xi,t)\dd\xi.
\end{split}
\]
Applying the Fourier splitting method, using the same trick as in~\eqref{FS-meth}
and the fact that $g(t)\le1$, we obtain the following version of \eqref{62bis} written for $z-\zl$:
\[
\begin{split}
\|(z-\zl)(t)\|_{L^2}^2- \|(z-\zl)(s)\|_{L^2}^2
&+c_0\int_s^t g(r)^2\|(z-\zl)(r)\|_{L^2}^2\dd r\\
&\le 
\frac{c_0}{4\pi}\int_s^t g(r)^2I_{u-u_L}(r)\dd r
+c_1\int_s^t (1+r)^{-1-2\Gamma}\dd r\\
&\le 
\frac{c_0}{4\pi}\int_s^t g(r)^2I_{z-z_L}(r)\dd r
+c_1\int_s^t (1+r)^{-1-2\Gamma}\dd r.
\end{split}
\]
Applying Lemma~\ref{Brando_Revista} with 
$e(t):=\exp(\int_{s_0}^t c_0g(s)^2\dd s)=\frac{(t+A)^\alpha}{(s_0+A)^\alpha}$,
we deduce that, for all $t\ge s_0>0$,
\begin{equation}
\|(z-z_L)(t)\|_{L^2}^2 \,e(t)
\lesssim 
1+\int_{s_0}^t e(s)g(s)^2I_{z-\zl}(s)\dd s+\int_{s_0}^t e(s)(1+s)^{-1-2\Gamma}\dd s.
\end{equation}

From Duhamel formula, since $z-\zl$ vanishes at $t=0$, we have
\begin{equation*}
(\widehat z-\widehat{\zl}\,)(\xi,t)=
 	\int_0^t  K(\xi,t-s)\mathcal{F}
  	\begin{pmatrix}
	-\P\nabla\cdot(u\otimes u)\\
	-\nabla\cdot(u\otimes h)
	\end{pmatrix}(\xi,s)
	\dd s.	
\end{equation*}
Hence, using the boundedness of $K$ (see Proposition \ref{sim-heat}),
\[
\begin{split}
|\widehat z-\widehat{\zl}|(\xi,t)
 	&\lesssim|\xi|\int_0^t 
 	\bigl(\|u(s)\|_{L^2}^2+\|u(s)\|_{L^2}\|h(s)\|_{L^2}\bigr)\dd s
\lesssim
 |\xi|\int_0^t (1+s)^{-\Gamma}\dd s.
\end{split}
\]
So
\[
I_{z-\zl}(t)\lesssim g(t)^4\Bigl(\int_0^t (1+s)^{-\Gamma}\dd s\Bigr)^2.
\]
Therefore, for all $t\ge s_0>0$,
\[
\|(z-\zl)(t)\|_{L^2}^2(t+A)^\alpha
\lesssim
1+\int_{s_0}^t (s+A)^{\alpha-3} 
\Bigl(\int_{0}^s(1+r)^{-\Gamma}\dd r \Bigr)^2\dd s
+\int_{s_0}^t(s+A)^{\alpha-1-2\Gamma}\dd s.
\]
So, the first part of Conclusion~ii) of Theorem~\ref{th:decay} follows.

\medskip
For the last assertion of Conclusion~ii), namely
the faster decay for $\|h-\hl\|_{L^2}^2$, we need to assume $\gamma>0$.
Applying Duhamel's principle to $h$ and $\hl$ using the  semigroup $e^{t(\gamma  \Delta-4\chi)}$, we obtain  
\begin{equation*}
\begin{split}
    h(t) &= e^{-4\chi\,t} e^{\gamma\,t\,\Delta} h_0 
    - \int_0^t e^{-4\chi(t-s)}\,e^{\gamma(t-s)\Delta}\,(u\cdot\nabla h)(s) \dd s 
    + 2\chi\, \int_0^t e^{-4\chi(t-s)}\,e^{\gamma(t-s)\Delta}\,\curl u (s)  \dd s,
    \\ 
     \hl(t) &= e^{-4\chi\,t} e^{\gamma\,t\,\Delta} h_0  
    + 2\chi\, \int_0^t e^{-4\chi(t-s)}\,e^{\gamma(t-s)\Delta}\,\curl \ul (s)  \dd s.
\end{split}
\end{equation*}
Therefore,
\begin{equation}\label{difference-h}
\begin{split}
    \|h(t)-\hl(t)\|_{L^2} 
    &\leq 
    \int_0^t e^{-4\chi(t-s)}\,\|e^{\gamma(t-s)\Delta}\,(u\cdot\nabla h)(s) \|_{L^2} \dd s 
    \\ 
    &\qquad\qquad
    + 2\chi\, \int_0^t e^{-4\chi(t-s)}\,\|e^{\gamma(t-s)\Delta}\,(\curl u - \curl \ul)(s) \|_{L^2} \dd s\\
       &=: \mathcal{J}_1(t) + \mathcal{J}_2(t) .
\end{split}
\end{equation}
To go further, we need to estimate $ \|\nabla z(t)-\nabla \zl(t)\|^2_{L^2}$.
We make use of the integral representation of $z(t)$ and 
$\zl(t)= e^{tM}z_0$. By the Duhamel formula applied to \eqref{MP2D} with the semigroup $e^{tM}$ we obtain
\[
\nabla z(t)-\nabla \zl(t)=\int_0^t \nabla e^{(t-s)M}
\begin{pmatrix}
	-\P\nabla\cdot(u\otimes u)\\
	-\nabla\cdot(u\otimes h)
	\end{pmatrix}(s)
	\dd s.
\]
Applying the $L^1$-$L^2$ and the $L^2$-$L^2$ estimates for $e^{tM}$ (see Proposition \ref{sim-heat} and relation \eqref{pl1}) we deduce
\begin{equation*}
    \begin{split}
        \|\nabla z(t)-\nabla \zl(t)\|_{L^2} 
        &\leq \sum_{j=1}^2\int_0^{\frac t2} \Bigl\|\nabla \partial_j e^{(t-s)M} 
        \begin{pmatrix}
	    \P(u_j u)\\
	    u_j h
	    \end{pmatrix}(s)
        \Bigr\|_{L^2} \dd s
+\sum_{j=1}^2\int_{\frac t2}^t\Bigl\| e^{(t-s)M} \begin{pmatrix}
	\P\partial_j(u\cdot\nabla u)\\
	\partial_j(u\cdot\nabla h)
	\end{pmatrix}(s) \Bigr\|_{L^2} \dd s
\\
        &\lesssim 
        \int_0^{t/2} (t-s)^{-3/2} \|z(s)\|_{L^2}^2 \dd s
         +\sum_{\ell=0}^1 \int_{t/2}^t \|\nabla^\ell z \|_{L^\infty} \|\nabla^{2-\ell} z \|_{L^2} \dd s\\
        &\lesssim 
        t^{-3/2}\int_0^{t/2} (1+s)^{-\Gamma} \dd s 
         +\sum_{\ell=0}^1 \int_{t/2}^t \|\nabla^\ell z \|_{L^\infty} \|\nabla^{2-\ell} z \|_{L^2} \dd s. 
            \end{split}
\end{equation*}

In the case $\gamma >0 $, the authors of \cite{Cesar-RG-Zingano-CP2023} proved the following result: under the assumption that 
$\|\zl(t)\|^2_{L^2} = O(t^{-\Gamma})$ as $t\to+\infty$, we have for all $j\in\N$
\begin{equation}\label{ref32}
\|\nabla^j\,z(t)\|^2_{L^2} = O(t^{-\Gamma - j})
\quad\text{and}\quad
\|\nabla\,h(t)\|^2_{L^2} = O(t^{-\Gamma - 2})
\quad\text{as }t\to \infty.
\end{equation}
Using the Gagliardo-Nirenberg inequality $\|f\|_{L^\infty}\leq C\|f\|_{L^2}^{\frac12}\|\nabla^2 f\|_{L^2}^{\frac12}$ we deduce that $\|\nabla^\ell z(t)\|_{L^\infty}=O(t^{-\Gamma/2-(\ell+1)/2})$ as $t\to \infty$.
Therefore, $\|\nabla^\ell z(t) \|_{L^\infty} \|\nabla^{2-\ell} z(t) \|_{L^2} = O(t^{-\Gamma-3/2})$ for $\ell = 0, 1$, which leads us to
\begin{equation*}
    \|\nabla z(t)-\nabla \zl(t)\|_{L^2}
    \leq C t^{-3/2}\int_0^{t/2} (1+s)^{-\Gamma} \dd s 
    + Ct^{-\Gamma- 1/2}
\end{equation*}
for $t$ sufficiently large.
This gives, for $t$ large enough,
$\|\nabla z(t)-\nabla \zl(t)\|^2_{L^2} \lesssim t^{-1}\zeta_\Gamma(t)$, 
where 
\[
\zeta_\Gamma(t):=
\begin{cases}
(1+t)^{-2\Gamma}& 0\le\Gamma< 1\\
(1+t)^{-2}(\log(e+t))^2& \Gamma=1\\
(1+t)^{-2}& 1<\Gamma\le 2.
\end{cases}
\]

We now turn our attention to the expression \eqref{difference-h}. Let us bound first $\mathcal{J}_1(t)$. Using once again \eqref{ref32}, the $L^\infty$ estimates proved above, the boundedness of $u$ and $h$ in $L^2$ and the basic properties of the heat semigroup, we have  
\begin{equation*}
\begin{split}
    \mathcal{J}_1(t) 
    &\lesssim \,e^{-2\chi\,t} \int_0^{t/2} (t-s)^{-1} \| u(s)\|_{L^2} \| h(s)\|_{L^2} \dd s 
    +     \int_{t/2}^t e^{-4\chi\,(t-s)}\,\| u(s)\|_{L^\infty} \|\nabla h\|_{L^2} \dd s \\
    &\lesssim 
    C\,e^{-2\chi\,t} + \int_{t/2}^t e^{-4\chi\,(t-s)}\,s^{-\Gamma - 3/2} \dd s
    \\
    &\lesssim t^{-\Gamma - 3/2}.
\end{split}
\end{equation*}
For the term $\mathcal{J}_2(t)$, similarly, we split the integral at $t/2$ and we obtain, for $t\ge1$,  
\begin{equation*}
\begin{split}
    \mathcal{J}_2(t) 
    &\lesssim
    e^{-2\chi\,t} \int_0^{t/2} (t-s)^{-1/2}\,\| z_0\|_{L^2} \dd s + 
     \int_{t/2}^t e^{-4\chi\,(t-s)}\,\| \nabla z(s) - \nabla \zl(s)\|_{L^2} \dd s     \\
    &\lesssim 
    \,t^{1/2}e^{-2\chi\,t} + \int_{t/2}^t e^{-4\chi\,(t-s)}\,s^{-1/2} \zeta_\Gamma(s)^{1/2}\dd s\\
    &\lesssim t^{-1/2}\zeta_\Gamma^{1/2}(t).
\end{split}
\end{equation*}
This proves estimate~\eqref{diff-h} for $t\ge1$ and the case $0\le t\le 1$ is settled
using that $\|h(t)\|_{L^2}+\|\hl(t)\|_{L^2}\lesssim \|z_0\|_{L^2}$.

Conclusion~ii) of Theorem~\ref{th:decay} follows.

\section{Various consequences of Theorem~\ref{th:decay} and conclusions}
\label{sec:concl}

\subsection*{Classes of data leading to algebraic decay}
Theorem~\ref{th:decay} is based on the assumption that $\|\ul\|_{L^2}^2$
has an algebraic decay:
\begin{equation}
\label{decayhyp}
\|u_L(t)\|_{L^2}^2\lesssim (1+t)^{-\Gamma}.
\end{equation}
Here we discuss some classes of initial data for which
such a condition holds.

Let us assume $(u_0,h_0)\in  L^2_\sigma(\R^2)\times L^2(\R^2)$.
A first elementary observation is that,
if the solution of the heat equation with data $z_0=(u_0,h_0)$
has an algebraic decay in the $L^2$-norm, then the solution $\zl$ to our linear problem
satisfies the same algebraic decay. 
Namely, if
\[
\|e^{t\Delta}z_0\|_{L^2}^2\lesssim (1+t)^{-\Gamma}
\]
then \eqref{decayhyp} holds true.
Indeed, applying Proposition \ref{sim-heat} for the symbol $K(\cdot,t)$, we see that
\[
\|\zl(t)\|_{L^2}^2\lesssim \|K(\cdot,t)\widehat z_0\|_{L^2}^2
\lesssim
 \|e^{ct\Delta}z_0\|_{L^2}^2+e^{-2ct}\|z_0\|_{L^2}^2
 \lesssim (1+t)^{-\Gamma}.
\]

But in fact, we can give a less restrictive condition. We observed (see the beginning of the proof  of Part i) of Theorem~\ref{th:decay} on page \pageref{startproof}) that \eqref{decayhyp} is equivalent to
\begin{equation}
\label{decayhyp2}
\|e^{\mu t\Delta}(u_0-\frac12\nabla^\perp h_0)\|_{L^2}^2=O(t^{-\Gamma})
\quad\text{as }t\to+\infty.
\end{equation}
Let us assume that
\begin{equation}
\label{decayhyp3}
\|e^{t\Delta}u_0\|_{L^2}^2\lesssim (1+t)^{-\Gamma}.
\end{equation}
By the decay estimates of the heat semigroup, we know that
\begin{equation*}
\|\nabla e^{\mu t\Delta}h_0)\|_{L^2}^2=O(t^{-1})\quad\text{as }t\to+\infty.
\end{equation*}
So in the case $0\leq\Gamma\leq1$, relation  \eqref{decayhyp3} implies \eqref{decayhyp2}, so it also implies \eqref{decayhyp}.

In the case $\Gamma\geq1$ we need to add the assumption
\begin{equation}
\label{decayhyp4}
\|e^{t\Delta}h_0\|_{L^2}^2\lesssim (1+t)^{-\Gamma+1}.
\end{equation}
Indeed, we then have
\begin{equation*}
\|\nabla e^{t\Delta}h_0\|_{L^2}^2
=\|\nabla e^{\frac t2 \Delta}e^{\frac t2 \Delta}h_0\|_{L^2}^2
\lesssim t^{-1}\|e^{\frac t2 \Delta}h_0\|_{L^2}^2
\lesssim t^{-1}(1+\frac t2)^{-\Gamma+1}
=O((1+t)^{-\Gamma})
\quad\text{as }t\to+\infty.
\end{equation*}

To summarize, we proved in the case $0\leq\Gamma\leq1$ that relation  \eqref{decayhyp3} implies \eqref{decayhyp}, and in the case  $\Gamma\geq1$ relations \eqref{decayhyp3} and \eqref{decayhyp4} imply \eqref{decayhyp}.

Now, it is well known that algebraic decays for the solutions of the heat
equations can be achieved if (and only if) the initial data belong to
Besov spaces with negative regularity.
(See~\cite[Chapter~2]{BahCD11}).
Therefore, the above mentioned implications can be reformulated as follows:
\begin{alignat*}{3}
\forall \,0\le \Gamma&\le1,\quad
u_0\in L^2_\sigma(\R^2)\cap \dot B^{-\Gamma}_{2,\infty}(\R^2),\;
h_0\in L^2(\R^2)
&\;\Longrightarrow\;
\eqref{decayhyp}\text{ and }\eqref{decayhyp2},\\
\forall \,\Gamma&\ge1,\quad
u_0\in L^2_\sigma(\R^2)\cap \dot B^{-\Gamma}_{2,\infty}(\R^2),\;
h_0\in L^2(\R^2)\cap \dot B^{-\Gamma+1}_{2,\infty}(\R^2)
&\;\Longrightarrow\;
\eqref{decayhyp}\text{ and }\eqref{decayhyp2}.
\end{alignat*} 
For example, elementary situations for which Theorem~\ref{th:decay} applies are the following:
\begin{enumerate}
\item 
\label{RR1}
If $u_0\in L^1(\R^2)\cap L^2_\sigma(\R^2)$ and $h_0\in 
L^2(\R^2)$, then
Theorem~\ref{th:decay} applies with $\Gamma=1$.
In general, if $1\le p\le2$,  $u_0\in L^p(\R^2)\cap L^2_\sigma(\R^2)$ and $h_0\in L^2(\R^2)$,
then Theorem~\ref{th:decay} applies with $\Gamma=\frac2p-1$.
\item
\label{RR2}
If $u_0\in  L^1(\R^2,(1+|x|)\dd x)\cap L^2_\sigma(\R^2)$
and $h_0\in L^1(\R^2)\cap L^2(\R^2)$, then Theorem~\ref{th:decay} applies with $\Gamma=2$.
\end{enumerate}
Indeed, claim~\eqref{RR1} follows from the usual $L^p$-$L^2$ heat kernel
estimates. Claim~\eqref{RR2} follows from the explicit representation of $e^{t\Delta}u_0$
in terms of the Gaussian 
$G_t(x)=(4\pi t)^{-1}\exp(-|x|^2/(4t))$, namely  
$e^{t\Delta}u_0(x)=\int G_t(x-y)u_0(y)\dd y$,
using also the fact that $\int u_0=0$ because of the divergence-free condition. (See e.g., \cite{Fujigaki-Miyakawa-SIAM}).

\subsection*{Asymptotics for solutions of the Micropolar system in the case
$\gamma>0$}

In this subsection, we assume the spin viscosity $\gamma$ to be strictly positive.
Let $0\le\Gamma\le2$ and let us consider an initial datum $z_0=(u_0,h_0)$, such that
    \begin{equation}
\label{Assug}
    (u_0,h_0)\in  L^2_\sigma(\R^2)\times L^2(\R^2)
\quad\text{and}\quad
    \|\ul\|_{L^2}^2\lesssim(1+t)^{-\Gamma}.
\end{equation}
See the previous subsection for explicit classes of initial data for which an $L^2$-algebraic decay for  $\ul$ holds.

From Theorem~\ref{th:linear} we obtain, for all $t\ge1$,
\[
\|u(t)-e^{\mu t\Delta}u_0+\frac12\nabla^\perp e^{\mu t\Delta}h_0\big\|_{L^2}^2
\lesssim
\|u(t)-\ul(t)\|_{L^2}^2+ t^{-2} \|u_0\|_{L^2}^2+t^{-3}\|h_0\|_{L^2}^2
\]
and
\[
\big\|h(t)- \frac12\curl e^{\mu t\Delta}u_0+\frac14\Delta e^{\mu t\Delta}h_0\big\|_{L^2}^2
\lesssim
\|h(t)-\hl(t)\|_{L^2}^2 + t^{-3} \|u_0\|_{L^2}^2+t^{-4} \|h_0\|_{L^2}^2.
\]
Let us recall that we set
\[
\zeta_\Gamma(t):=
\begin{cases}
(1+t)^{-2\Gamma}& 0\le\Gamma< 1\\
(1+t)^{-2}(\log(e+t))^2& \Gamma=1\\
(1+t)^{-2}& 1<\Gamma\le 2.
\end{cases}
\]
By the last part of Theorem~\ref{th:decay},  for $t\ge1$,
\begin{equation}
\label{corasy}
\begin{split}
\|u(t)-e^{\mu t\Delta}u_0
 +\frac12\nabla^\perp e^{\mu t\Delta}h_0\|_{L^2}^2 
&\lesssim \zeta_\Gamma(t)\\
\big\|h(t)- \frac12\curl e^{\mu t\Delta}u_0
+\frac14\Delta e^{\mu t\Delta}h_0\big\|_{L^2}^2
&\lesssim t^{-1}\zeta_\Gamma(t).
\end{split}
\end{equation}
Let us now distinguish several cases, depending on the range of $\Gamma$.

\subsubsection*{Case $0<\Gamma\le\frac12$.}

We have in this case
 $\|\nabla^\perp e^{\mu t\Delta}h_0\|_{L^2}^2=O(t^{-1})=O(\zeta_\Gamma(t))$
  as $t\to+\infty$.
So this term can be incorporated inside the remainder of expansion~\eqref{corasy}.
In the same way, $\|\frac14\Delta e^{\mu t\Delta}h_0\|_{L^2}^2=O(t^{-2})
=O(t^{-1}\zeta_\Gamma(t))$, as $t\to+\infty$,
so this term can also be included inside the remainders.
Hence, for any finite energy solution $(u,h)$  to the nonlinear problem~\eqref{MP2D} that arises from a datum~$(u_0,h_0)$
satisfying~\eqref{Assug},
we obtain
\[
\begin{split}
\|u(t)-e^{\mu t\Delta}u_0\|_{L^2}^2
&\lesssim 
(1+t)^{-2\Gamma}.\\
\big\|h(t)- \frac12\curl e^{\mu t\Delta}u_0 \big\|_{L^2}^2
&\lesssim
(1+t)^{-2\Gamma-1}.
\end{split}
\]
In other words, we obtain the linear asymptotics 
\begin{subequations}
\label{expan}
\begin{equation}
u(t)\stackrel{t\to+\infty}{\approx} e^{\mu t\Delta}u_0
\qquad
\text{and}
\qquad
h(t)\stackrel{t\to+\infty}{\approx} \frac12\curl e^{\mu t\Delta}u_0
\end{equation}
with remainders of order $O(t^{-\Gamma})$, respectively  $O(t^{-\Gamma-\frac12})$, in the $L^2$-norm. For such slow decay rates, our results do not allow us to go beyond this first-order asymptotics for the velocity component
of the solution to the nonlinear problem~\ref{MP2D}.

\subsubsection*{Case $\frac12<\Gamma< 2$.}
In this case,
applying~\eqref{corasy}, we obtain a second-order asymptotics in $L^2$
for the velocity component of a finite energy solution to~\eqref{MP2D}:
\begin{equation}\label{expan2}
u(t)\stackrel{t\to+\infty}{\approx} e^{\mu t\Delta}u_0 -\frac12\nabla^\perp e^{\mu t\Delta}h_0
\qquad
\text{and}
\qquad
h(t)\stackrel{t\to+\infty}{\approx} \frac12\curl e^{\mu t\Delta}u_0
-\frac14\Delta e^{\mu t\Delta}h_0
\end{equation}
\end{subequations}
with remainders of order $O(t^{-\min(1,\Gamma)})$, respectively  $O(t^{-\min(1,\Gamma)-\frac12})$ (with an additional logarithmic factor if $\Gamma=1$), in the $L^2$-norm.
\subsubsection*{Case $\Gamma\ge2$.}
When $\Gamma\ge2$ and the data satisfy~\eqref{Assug}, the results of this paper do not allow to go beyond \eqref{expan2} and find the next term in the asymptotic expansion as $t\to\infty$. We still have  \eqref{expan2} with remainders of order $O(t^{-1})$, respectively  $O(t^{-\frac32})$ but nothing more. The derivation of a \emph{nonlinear asymptotic profile} for the solution
$(u,h)$ to~\eqref{MP2D} in this fast-decaying regime is an interesting problem to address, but it is outside the scope of this paper.
It is easy to construct examples of initial data satisfying~\eqref{Assug}
with $\Gamma\gg2$, by considering initial data
such that $(\widehat u_0,\widehat h_0)(\xi)$ vanish at $\xi=0$ at a sufficiently high order and relying on the observations from the previous paragraph. But
despite the fast decay of the linear solution  $\zl$, the corresponding solutions $z=(u,h)$ of the nonlinear problem~\eqref{MP2D} are not expected to decay in the $L^2$-norm faster than $O(t^{-1})$ as $t\to+\infty$, in general.

\begin{remark}
Under more stringent conditions on the data, one could go further and make even more explicit the asymptotic expansions~\eqref{expan}, by using the classical results available for describing the large time behavior of 
the solutions of the heat equation. See, e.g., \cite{Zuazua1}.
For example, 
if $h_0\in L^1(\R^2)$, and $\int h_0\not=0$, then one can use the well known fact that
\[
\textstyle
\nabla^\perp e^{\mu t\Delta}h_0\stackrel{t\to+\infty}\approx 
(\int h_0)\nabla^\perp G_{\mu t}(\cdot),
\qquad
\text{and}
\qquad
\frac14\Delta e^{\mu t\Delta}h_0
\stackrel{t\to+\infty}\approx \frac14(\int h_0)\Delta G_{\mu t}(\cdot).
\]
in the $L^2$-norm, where $G_t(x)=(4\pi)^{-1}\exp(-|x^2|/(4t))$,
is the Gaussian kernel.
It would be tempting to make use of a similar observation for expanding
$e^{\mu t\Delta}u_0$, but one should take into account that,
if $u_0\in L^1(\R^2)$, then the divergence-free condition implies that
$\int u_0=0$: for this reason, different conditions on~$u_0$ (e.g., a prescribed behavior for $\widehat u_0(\xi)$ as $\xi\to0$), are better suited to single out the main term of the asymptotics of $e^{\mu t\Delta}u_0$.
\end{remark}

\subsection*{Asymptotics for solutions of the Micropolar system in the case
$\gamma=0$}

When $\gamma=0$, our assumptions on the initial data need to be more stringent. In addition to condition~\eqref{Assug}, we also require that $\int(1+|\xi|)|\widehat z_0(\xi)|\dd\xi<\infty$. With this additional assumption, we have the same asymptotics for the velocity $u$ as in the case $\gamma>0$, but the asymptotic results will be weaker for $h$. Instead of~\eqref{corasy}, in the case $\gamma=0$ the application of Theorem~\ref{th:decay}
implies the following weaker estimate for $h$:
\begin{equation}
\label{corasy0}
\begin{split}
\big\|h(t)- \frac12\curl e^{\mu t\Delta}u_0&\big\|_{L^2}^2
\lesssim 
\zeta_\Gamma(t).
\end{split}
\end{equation}
Notice that we dropped the term $\frac14\Delta e^{\mu t\Delta}h_0$, because 
$\|\frac14\Delta e^{\mu t\Delta}h_0\|_{L^2}^2=O(t^{-2})$ can always be included in the remainders.

\subsubsection*{Case $0\le\Gamma\le\frac12$.}
Since
$\|\curl e^{\mu t\Delta}u_0\|_{L^2}^2=O(t^{-1})$,
we have that \eqref{corasy0} is equivalent to
\[
\|h(t)\|_{L^2}^2 \lesssim(1+t)^{-2\Gamma}, 
\]
but, contrary to $u$, no explicit asymptotic profile is available for $h$ in this slow decay regime, when $\gamma=0$.

\subsubsection*{Case $\frac12<\Gamma\leq2$.}
In this range, we obtain a first order asymptotics for $h$:
\[
h(t)\approx \textstyle\frac12\curl e^{\mu t\Delta}u_0
\quad\text{in $L^2$}
\]
with a remainder of order $O(t^{-\min(1,\Gamma)})$ (with an additional logarithmic factor if $\Gamma=1$).

\bigskip
 
\subsection*{Further consequences}
An interesting byproduct of our analysis is the existence of solutions
to the micropolar equations \emph{dissipating their energy faster} than the
corresponding solutions of the Navier-Stokes system, with the same
initial velocity $u_0$. In other words, there are situations in which
the presence of microrotational effects considerably enhances the dissipation,
slowing down the fluid motion for large times.
For example, let $1<\Gamma<2$, and
let us consider the initial datum $u_0\in L^2_\sigma(\R^2)$, defined by: 
\[
\widehat u_0(\xi)=\frac12\begin{pmatrix}-i\xi_2\\i\xi_1\end{pmatrix}
|\xi|^{-2+\Gamma}\Phi(\xi),
\]
where $\Phi$ is an arbitrary function, continuous in $\R^2$, such that $\Phi(0)\not=0$ and $\Phi(\xi)=O(|\xi|^{-2})$
as $|\xi|\to+\infty$. The latter condition is just to ensure that $\widehat u_0\in L^2(\R^2)$.
Then the solution of the heat equation satisfies (see, e.g., \cite[Theorem~3.2]{Bra-Per-Zin}):
\[
 (1+t)^{-\Gamma}\lesssim \|e^{\mu t\Delta}u_0\|_{L^2}^2 \lesssim (1+t)^{-\Gamma}.
\]
By Wiegner's theorem, see~\cite{Wie87}, the solution to the Navier-Stokes $u_{\text{NS}}$ arising from~$u_0$
satisfies
\begin{equation}
\label{uNS}
 (1+t)^{-\Gamma}\lesssim \|u_{\text{NS}}(t)\|_{L^2}^2 \lesssim (1+t)^{-\Gamma}.
\end{equation}

Now, let us assume that $h_0\in L^2(\R^2)$ has the form 
\[
\widehat h_0(\xi)=|\xi|^{-2+\Gamma}\Phi(\xi).
\]

Clearly $u_0-\frac12\nabla^\perp h_0=0$ so relation \eqref{ul11} implies that $\|\ul(t)\|_{L^2}^2\lesssim (1+t)^{-2}$. Moreover, we can apply Theorem \ref{th:decay} for $\Gamma=2$ and obtain that $\|u(t)-\ul(t)\|_{L^2}^2\lesssim (1+t)^{-2}$. So $\|u(t)\|_{L^2}^2\lesssim (1+t)^{-2}$ which, in view of \eqref{uNS}, shows that $u(t)$ decays indeed faster than $u_{\text{NS}}(t)$ in the $L^2$-norm.

\vspace{.5cm}

\scriptsize{
\textbf{Acknowledgments.}  
C. F. Perusato was partially supported by CNPq grant No. 200124/2024-2 and CNPq research fellowship No. 310444/2022-5. He would also like to express his gratitude for the warm hospitality and fruitful discussions during his one-year stay at Université Lyon 1 (Institut Camille Jordan), where this work was done.
Research partially supported by CAPES/COFECUB project N.1040-24 WINPDE. L. Brandolese was partially supported also by ANR-25-CE40-4532.

\medskip

\begin{description}
\item[Lorenzo Brandolese] Institut Camille Jordan. 
		Université de Lyon, Université Claude Bernard Lyon~1, 
		69622 Villeurbanne Cedex. France\\
Email: \texttt{brandolese@math.univ-lyon1.fr}
\item[Adriana Valentina Busuioc] Université Jean Monnet, CNRS, Centrale Lyon, INSA Lyon, Universite Claude Bernard Lyon 1, ICJ UMR5208, 42023 Saint-Etienne, France. \\
Email: \texttt{valentina.busuioc@univ-st-etienne.fr}\\
Web page: \texttt{https://perso.univ-st-etienne.fr/busuvale/}
\item[Dragoş Iftimie] Institut Camille Jordan. 
		Université de Lyon, Université Claude Bernard Lyon~1, 
		69622 Villeurbanne Cedex. France\\
Email: \texttt{iftimie@math.univ-lyon1.fr}\\
Web page: \texttt{http://math.univ-lyon1.fr/\~{}iftimie}
\item[Cilon F. Perusato] Departamento de Matem\'atica. 
		Universidade Federal de Pernambuco, Recife, PE 50740-560. Brazil \\ and Institut Camille Jordan, 
		Université de Lyon, Université Claude Bernard Lyon~1, 
		69622 Villeurbanne Cedex. France
Email: \texttt{perusato@math.univ-lyon1.fr}
\end{description}


\begin{thebibliography}{10}

\bibitem{BahCD11}
H.~Bahouri, J.-Y. Chemin and R.~Danchin.
\newblock {\em Fourier analysis and nonlinear partial differential equations},
  volume 343 of {\em Grundlehren der Mathematischen Wissenschaften [Fundamental
  Principles of Mathematical Sciences]}.
\newblock Springer, Heidelberg, 2011.

\bibitem{Brando-Revista}
L.~Brandolese.
\newblock Asymptotic behavior of the energy and pointwise estimates for
  solutions to the Navier-Stokes equations.
\newblock {\em Rev. Mat. Iberoamericana}, 20:223--256, 2004.

\bibitem{Bra-Per-Zin}
L.~Brandolese, C.~F. Perusato and P.~R. Zingano.
\newblock On the topological size of the class of Leray solutions with
  algebraic decay.
\newblock {\em Bull. Lond. Math. Soc.}, 56(1):59--71, 2024.

\bibitem{BCFZ}
P.~Braz~e Silva, F.~W. Cruz, L.~B.~S. Freitas and P.~R. Zingano.
\newblock On the $L^2$ decay of weak solutions for the 3D asymmetric fluids
  equations.
\newblock {\em J. Differential Equations}, 267(6), 2019.

\bibitem{Chen-Price}
Z.-M. Chen and W.~G. Price.
\newblock Decay estimates of linearized micropolar fluid flows in
  $\mathbb{R}^3$ space with applications to $L_3$-strong solutions.
\newblock {\em Internat. J. Engrg. Sci.}, 44(13-14):859--873, 2006.

\bibitem{diperna_ordinary_1989}
R.~J. DiPerna and P.-L. Lions.
\newblock Ordinary differential equations, transport theory and Sobolev spaces.
\newblock {\em Inventiones Mathematicae}, 98(3):511--547, 1989.

\bibitem{Dong-Zhang}
B.~Dong and Z.~Zhang.
\newblock Global regularity of the 2D micropolar fluid flows with zero angular
  viscosity.
\newblock {\em J. Differ. Equ.}, 249:200--213, 2010.

\bibitem{Ericksen1990}
J.~L. Ericksen.
\newblock Liquid crystals with variable degree of orientation.
\newblock {\em Arch. Rational Mech. Anal.}, 113(2):97--120, 1990.

\bibitem{Eringen}
A.~C. Eringen.
\newblock Theory of micropolar fluids.
\newblock {\em J. Math. Mech.}, 16:1--18, 1966.

\bibitem{Fujigaki-Miyakawa-SIAM}
Y.~Fujigaki and T.~Miyakawa.
\newblock Asymptotic profiles of nonstationary incompressible Navier-Stokes
  flows in the whole space.
\newblock {\em SIAM J. Math. Anal.}, 33(3):523--544, 2001.

\bibitem{Galdi-Rionero}
G.~P. Galdi and S.~Rionero.
\newblock A note on the existence and uniqueness of solutions of micropolar
  fluid equations.
\newblock {\em Internat. J. Engrg. Sci.}, 14:105--108, 1977.

\bibitem{GW05}
T.~Gallay and C.~E. Wayne.
\newblock Global stability of vortex solutions of the two-dimensional
  Navier-Stokes equation.
\newblock {\em Comm. Math. Phys.}, 255(1):97--129, 2005.

\bibitem{GRT}
F.~Gay-Balmaz, T.~S. Ratiu and C.~Tronci.
\newblock Equivalent theories of liquid crystal dynamics.
\newblock {\em Arch. Ration. Mech. Anal.}, 210(3):773--811, 2013.

\bibitem{Guo-Jia-Dong}
Y.~Guo, Y.~Jia and B.-Q. Dong.
\newblock Time Decay Rates of the Micropolar Equations with Zero Angular
  Viscosity.
\newblock {\em Bulletin of the Malaysian Mathematical Sciences Society}, 2021.

\bibitem{Cesar-RG-Zingano-CP2023}
R.~Guterres, W.~Melo, C.~Niche, C.~Perusato and P.~Zingano.
\newblock Strong Alignment of Micro-rotation and Vorticity in 3D Micropolar
  Flows.
\newblock {\em Nonlinearity}, 38(015006), 2025.

\bibitem{GuterresNichePerusatoZingano2023}
R.~H. Guterres, C.~J. Niche, C.~F. Perusato and P.~R. Zingano.
\newblock Upper and lower $\dot H^m$ estimates for solutions to parabolic
  equations.
\newblock {\em J. Differential Equations}, 356:407--431, 2023.

\bibitem{GuterresNunesPerusato2019}
R.~H. Guterres, J.~R. Nunes and C.~F. Perusato.
\newblock On the large time decay of global solutions for the micropolar
  dynamics in $L^2(\mathbb{R}^n)$.
\newblock {\em Nonlinear Anal. Real World Appl.}, 45:789--798, 2019.

\bibitem{Kato1984}
T.~Kato.
\newblock Strong $L\sp{p}$-solutions of the Navier-Stokes equation in ${\bf
  R}\sp{m}$, with applications to weak solutions.
\newblock {\em Math. Z.}, 187(4):471--480, 1984.

\bibitem{ZinganoJMFM}
H.-O. Kreiss, T.~Hagstrom, J.~Lorenz and P.~Zingano.
\newblock Decay in time of incompressible flows.
\newblock {\em J. Math. Fluid Mech.}, 5(3):231--244, 2003.

\bibitem{Ladyzhenskaya}
O.~A. Ladyzhenskaya.
\newblock {\em The Mathematical Theory of Viscous Incompressible Fluids}.
\newblock Gorden Brech, New York, 1969.

\bibitem{Leray1934}
J.~Leray.
\newblock Sur le mouvement d'un liquide visqueux emplissant l'espace.
\newblock {\em Acta Math.}, 63(1):193--248, 1934.

\bibitem{lions_quelques_1969}
J.-L. Lions.
\newblock {\em Quelques m{\'e}thodes de r{\'e}solution des probl{\`e}mes aux
  limites non lin{\'e}aires}.
\newblock Dunod, 1969.

\bibitem{lions_mathematical_1996}
P.-L. Lions.
\newblock {\em Mathematical Topics in Fluid Mechanics. {{Vol}}. 1}, volume~3 of
  {\em Oxford {{Lecture Series}} in {{Mathematics}} and Its {{Applications}}}.
\newblock The Clarendon Press, Oxford University Press, New York, 1996.

\bibitem{Lukaszewicz}
G.~Lukaszewicz.
\newblock {\em Micropolar fluids. Theory and applications}.
\newblock Boston: Birkh\"auser, 1999.

\bibitem{Masuda1984}
K.~Masuda.
\newblock Weak solutions of Navier-Stokes equations.
\newblock {\em Tohoku Math. J. (2)}, 36(4):623--646, 1984.

\bibitem{Niu-Shang}
D.~Niu and H.~Shang.
\newblock Lower and upper bounds of decay to the d-dimensional
  magneto-micropolar equations.
\newblock {\em J. Math. Phys.}, 65:121501, 2024.

\bibitem{putzer_avoiding_1966}
E.~J. Putzer.
\newblock Avoiding the {{Jordan Canonical Form}} in the {{Discussion}} of
  {{Linear Systems}} with {{Constant Coefficients}}.
\newblock {\em The American Mathematical Monthly}, 73(1):2--7, 1966.

\bibitem{Rojas-Medar}
M.~A. Rojas-Medar.
\newblock Magneto-micropolar fluid motion: Existence and uniqueness of strong
  solution.
\newblock {\em Math. Nachr.}, 188:301–319, 1997.

\bibitem{Sch85}
M.~E. Schonbek.
\newblock $L^2$ decay for Weak Solutions of the Navier--Stokes Equations.
\newblock {\em Arch. Rat. Mech. Anal.}, 88:209--222, 1985.

\bibitem{Shliomis2021}
M.~I. Shliomis.
\newblock How a rotating magnetic field causes ferrofluid to rotate.
\newblock {\em Phys. Rev. Fluids}, 6(043701), 2021.

\bibitem{taylor_partial_1997}
M.~E. Taylor.
\newblock {\em Partial differential equations. {{III}}}.
\newblock Springer-Verlag, New York, 1997.

\bibitem{Temam}
R.~Temam.
\newblock {\em Navier–Stokes Equations, Theory and Numerical Analysis}.
\newblock North-Holland, Amsterdam, New York, 1977.

\bibitem{Wie87}
M.~Wiegner.
\newblock Decay results for weak solutions of the Navier-Stokes equations on
  ${\bf R}^n$.
\newblock {\em J. London Math. Soc. (2)}, 35(2):303--313, 1987.

\bibitem{Zuazua1}
E.~Zuazua.
\newblock Asymptotic behavior of scalar convection-diffusion equations.
\newblock {\em arXiv.org:2003.11834}, 2020.

\end{thebibliography}
\end{document}